\documentclass[11pt]{amsart}
\usepackage{yfonts,amsmath,amsfonts,amssymb,amsthm}
\usepackage[hidelinks,hyperindex,breaklinks]{hyperref}
\usepackage{yfonts,amsmath,amsfonts,amssymb,amsthm}
\usepackage{graphics}
\usepackage{epsfig}
\usepackage{esint}
\usepackage{color}
\newcommand{\ignore}[1]{}

\newtheorem{theorem}{Theorem}
\newtheorem{remark}{Remark}
\newtheorem{lemma}{Lemma}
\newtheorem{corollary}{Corollary}

\title[Quantitative homogenization of curves in a Brownian
potential]{Quantitative homogenization of the maximal action of curves in a Brownian
potential}

\author{Felix Otto, Matteo Palmieri}

\begin{document}

\begin{abstract}

Motivated by an optimal-matching problem (Leighton-Shor) and the random-field Ising model
(Aizenman-Wehr, Ding-Wirth),
we consider a variational problem for graphs in $1+1$ dimension maximizing
an action that is the difference of a field term given by 
integrating white noise over the subgraph on the one hand,
and the Dirichlet integral of the (continuum) height function $h=h(x)$ on the other hand.
This problem is scale-invariant in law, and requires a small-scale cut-off which we
implement by restricting to $h$ that are piecewise linear on intervals of size $1$
and vanish at $x=0,L$. 
We show that with overwhelming probablity, the maximal action $A$ satisfies $A=a\ln L+O(1)$ 
for a deterministic constant $a\in(0,\infty)$.
This can be considered as a homogenization result that is quantitative in an optimal way.  

\smallskip

The present result
sharpens a recent qualitative homogenization result by the authors with C.~Wagner;
it does so by finding bounds for the action that are locally uniform in the boundary conditions.
Like the earlier result, the present one relies on pointwise
bounds on the optimizer, as provided by Dembin-Elboim-Hadas-Peled in a more general setting. 
In the earlier work, the small-scale cut-off also involved an explicit discretization 
of the field term, yielding a Brownian potential that is i.~i.~d.~in $x\in\mathbb{Z}$;
this had the benefit of allowing for a comparison arguments, 
but is inconvenient for the coarse-graining used here.

\end{abstract}

\maketitle


\tableofcontents

\section{The main result}

\subsection{Problem setting: action and configuration space}\label{SS:setting}
For a function $h\colon[0,L]\to\mathbb{R}$ on the interval $[0,L]$
of dyadic size $L\in 2^\mathbb{N}$, we consider its Dirichlet energy
\begin{align}\label{e45}
D(h):=\frac{1}{2}\int_0^Ldx(\frac{dh}{dx})^2,
\end{align}
and the integral of white noise $\xi$ on $\mathbb{R}^2$ against its oriented subgraph
\begin{align}\label{e71}
W(h):=\int_0^Ldx\int_0^{h(x)}dy~\xi.
\end{align}
Note that $\int_a^bdx\int_0^h dy~\xi$ $=_{\rm law}\sqrt{b-a}\,B(h)$ where $B$
is two-sided Brownian motion.
We are interested in maximizing the
\begin{align}\label{e1}
\mbox{action}\quad W(h)-D(h)
\end{align}
over the
\begin{align}\label{e9}
\lefteqn{\mbox{configuration space}\quad
\big\{\,h\colon[0,L]\rightarrow\mathbb{R}\;\mbox{continuous}\,|}\nonumber\\
&\,\mbox{$h$ is linear
on $[n-1,n]$ for $n=1,\cdots,L$ with $h(0)=h(L)=0$}\,\big\}
\end{align}
of piecewise linear functions with homogeneous Dirichlet boundary conditions.
Note that on this $(L-1)$-dimensional space, not only $D$ but also $W$ are
continuous for almost every configuration of $\xi$, so that the maximization 
problem is well-defined (almost surely).


\subsection{Scale invariances in law}

Note that under the rescaling of $x$ and $y$ axes
\begin{align*}
(x,y)=(\lambda\hat x,\mu\hat y)\quad\mbox{and thus}\quad
L=\lambda\hat L,\;\;h=\mu\hat h,
\end{align*}
the two functionals transform as follows
\begin{align}\label{ao68}
D=(\mu^2/\lambda)\,\hat D\quad\mbox{and}\quad W=_{\rm law}\sqrt{\lambda\mu}\,\hat W.
\end{align}
In view of the different homogeneity in $\mu$, 
there is no loss in generality in having no relative constant in $W-D$, 
which amounts to an appropriate choice of units for $y$.
In view of the identical homogeneity for $\lambda=\mu$,
there is no loss in generality in choosing units for $x$ such that one of the length scales
of the problem is unity. Here, we opted to take the microscopic scale, i.~e.~the length
of the interval on which a configuration is linear, as unit. 


\subsection{Statement of main result}

The scaling (\ref{ao68}) in its isotropic variant $\lambda=\mu$ shows that the action $W-D$
has units of length so that it is natural to look at action $(W-D)/L$ per system length $L$.
Our main result, the upcoming Theorem \ref{T:1}, 
identifies the asymptotic behavior of the maximal action per system length
\begin{align}\label{ao15}
A_L:=\sup_{\text{configuration}\;h}(W-D)(h)/L;
\end{align}
it diverges logarithmically in $L$. Beyond scaling, we can show that there is a 
deterministic constant $a^*\in(0,\infty)$ such that
\begin{align}\label{ao69}
A_L=a^*\ln L+O(1)
\end{align}
with overwhelming probability, which can be seen as an (optimally) quantified 
homogenization result. Of course, the argument of the logarithm ought to be
interpreted as the ratio of the macroscopic scale as given by the system size $L$,
and the microscopic scale given by the interval size on which the configurations
are linear, which w.~l.~o.~g.~we set to one. In particular, the divergence as $L\uparrow\infty$
shows that the small-scale cut-off provided by the linearity assumption cannot be
easily removed.

\medskip

While for fixed configuration, $W(h)$ is Gaussian, it will be convenient to have
more flexibility in measuring the fluctuations in (\ref{ao69}), which is provided
by the family of Orlicz norms: For $s\in[1,\infty)$ and a random variable $X$ 
\begin{align}\label{ao74}
\|X\|_s:=\inf\{N>0:\mathbb{E}\exp(\frac{|X|}{N})^s\le e\}.
\end{align}
See Lemma~\ref{L:11} in the appendix for the relation of these norms with tail
decay. Note that they are normalized such that $|\mathbb{E}X|$ $\le\|X\|_s$.

%

\begin{theorem}\label{T:1}
There exists a deterministic constant $a^*\in(0,\infty)$ such that
\begin{align*}
\big\|A_L-a^*\ln L\big\|_{3/2}\lesssim 1,
\end{align*}
where $\lesssim$ is short for $\le C$ for some universal constant.
\end{theorem}

The order one error is optimal, as follows from \cite[(2) in Theorem 1.2]{DEP}.
It suggests that a renormalization of the problem that allows to send the small-scale
cut-off to zero is possible.
Such a scaling limit is known for (semi-continuous Brownian) last passage percolation, and
is described by what is called the directed landscape \cite[Theorem
1.5]{DOV}.
There are obvious differences: Loosely speaking, here, the white noise is integrated
over the subgraph while there it is integrated over the graph.
Hence while there, the optimal curve is not better than H\"older continuous 
with exponent $2/3$ in line with the scaling of the KPZ universality class, 
here it is Lipschitz continuous up to a logarithmic correction \cite[Proposition
3]{OPW}.
Moreover, there is no analogue from integrable probability like the 
Tracy-Widom distribution and the Airy process in sight.
However, there are last passage percolation models that due to i.~i.~d.~weights with
specific thick tails have the same critical scaling as ours. 
In these, the same exponent of the logarithm emerges \cite[Theorem 2]{G,GGN}.
We don't know whether the optimizer significantly changes when re-sampling a small fraction
of the randomness, a property dubbed disorder chaos by S.~Chatterjee, and recently
established for an isotropic version of first-passage percolation \cite[Section
2.4]{BSS}.
If the contrary were true, a renormalization would be at arm's length.

\subsection{Refinement of the main result: 
equipartition of energy}\label{SS:refine}

The origin of the logarithmic scaling becomes transparent by the upcoming Theorem \ref{T:2},
which refines Theorem \ref{T:1}, and can be interpreted as follows: 
Each of the $\log_2 L$ dyadic scales contributes the same amount to the action maximum.
To express this, for every intermediate dyadic scale $1\le l\le L$, 
we introduce a projection from configuration space $h\mapsto h_{\ge l}$ 
onto the space of coarse-grained configurations of scales $\ge l$,
as conveniently defined by linear interpolation at the mesoscopic grid $\{nl\}_{n=0,\cdots,L/l}$:
\begin{align}
&h_{\ge l}\colon[0,L]\rightarrow\mathbb{R}~\mbox{continuous and}\nonumber\\
&\mbox{linear on
$[(n-1)l,nl]$for $n=1,\cdots,L/l$}\nonumber\\
&\mbox{as determined by}\quad h_{\ge l}(nl)=h(nl)\quad\mbox{for $n=0,\cdots,L/l$};\label{m43}
\end{align}
we note that $h_{\ge 1}=h$ while $h_{\ge L}=0$.
With the notation
\begin{align*}
h^*:=\mbox{maximizer of $(W-D)/L$ on configuration space},\quad
\end{align*}
so that $(W-D)(h^*)/L=A_L$, and for two dyadic scales $1\le l'< l\le L$, 
we may interpret $(W-D)(h^*_{\ge l'})/L-(W-D)(h^*_{\ge l})/L$ 
as the amount the scales between $l'$ and $l$ contribute to $A_L$.

\begin{theorem}\label{T:2}
For any pair of dyadic scales $1\le l'< l\le L$ we have
\begin{align*}
\big\|(W-D)(h^*_{\ge l'})/L-(W-D)(h^*_{\ge l})/L-a^*\ln l/l'\big\|_{3/2}\lesssim 1.
\end{align*}
\end{theorem}

We recover Theorem \ref{T:1} for $(l',l)$ $=(1,L)$. A first motivation for this
additivity of $A_L$ in $\ln L$ will be given in Subsection \ref{SS:OPW}.

\medskip

Let us explain the relation with the previous paper \cite{OPW}, where
the close model with $W$ in \eqref{e71} replaced by
\begin{align*}
\tilde
W(h):=\sum_{n=1}^{L-1}W_n(h(x))\quad\mbox{with}\,\{W_n(\cdot)\}_n\,\mbox{independent
Brownian motions}
\end{align*}
was studied. With the aim of getting the optimal
convergence rate -- which is the main improvement with respect to that work --
we adopted the new field term in \eqref{e71}, in order to take advantage of the
invariance under isotropic rescaling ($\lambda=\mu$ in \eqref{ao68}). We also find the new
definition as being more natural. The modification forced us to change the proof
in order to avoid the maximum principle, which is not satisfied anymore by the
new functional and was previously used to get uniform (in the boundary
conditions) $L^\infty$ estimates on the minimizer. As a consequence, the new
strategy might be applied to higher-order Dirichlet energies and higher
codimensions.

\section{Motivation and context}

\subsection{Leighton-Shor on the optimal matching problem}\label{SS:LS}

We are given a cloud of points $\{X\}\subset B$ in a ball $B$ 
in the Euclidean space of dimension $d$,
with number $N=\#\{X\}$ equal to $\lambda(B)$, where
\begin{align*}
\mbox{$\lambda$ is the Lebesgue measure on the $d$-dimensional Euclidean space.}
\end{align*}
Consider a partition $\{F_X\subset B\}_X$, 
indexed by the points $X$, of $B$ into $N$ sets of unit volume
-- this may be considered a matching.
We monitor the maximum over $X$ of the distance between $X$ and the complement of $F_X$,
that is, $\max_X\max_{y\in F_X}|y-X|$ -- this is a matching distance. 
What is the infimum of this distance over all partitions $\{F_X\}_X$? 
Clearly, this optimal matching problem can be rephrased
as an optimal transportation problem between the atomic measure
\begin{align*}
\mu:=\sum_{X}\delta_X
\end{align*}
and the Lebesgue measure $\lambda\lfloor B$ (restricted to $B$), 
where the cost functional is the (essential) supremum of the displacement.

\medskip

Let us give some geometric intuition. For any subset $\Sigma\subset B$
consider the ``discrepancy'' $(\mu-\lambda)(\Sigma)$ that measures the excess
of number of points $\{X\}$ within $\Sigma$ over its volume. It is obvious that if $R$ is $\ge$ 
the above optimal matching distance, then there is enough volume at distance $R$
of $\Sigma$ to accommodate this excess:
\begin{align}\label{ao21}
(\mu-\lambda)(\Sigma)\le\lambda(\{x\in B-\Sigma|{\rm dist}(x,\Sigma)\le R\}).
\end{align}
The r.~h.~s.~ of (\ref{ao21}) is the (Lebesgue) measure of a tubular neighborhood of $\Sigma$;
it thus connects measure and (Euclidean) distance. 
As we will use later, it provides a definition\footnote{known as the Minkowski content}
of the (Hausdorff) measure of the boundary
$\partial\Sigma\cap B$, 
which is also called the perimeter of $\Sigma$ in $B$:
\begin{align}\label{ao22}
P(\Sigma,B)=
\lim_{r\downarrow 0}\frac{1}{r}\lambda(\{x\in B-\Sigma|{\rm dist}(x,\Sigma)\le r\}). 
\end{align}
Note that (\ref{ao21}) is only seemingly asymmetric: Replacing $\Sigma\subset B$ 
by its complement $B-\Sigma$, we obtain a lower bound because of
$(\mu-\lambda)(B-\Sigma)$ $=-(\mu-\lambda)(\Sigma)$.
It is intuitive that (\ref{ao21}) defines the variational problem dual to the
optimal matching problem,
a worst case~/~ma\-xi\-mi\-zation problem over subsets $\Sigma\subset B$, 
rather than a minimization problem over partitions $\{F_X\}_X$.

\medskip

We are interested in point clouds $\{X\}$ that are ``maximally'' random, 
more precisely, where the positions $X$ are independently chosen from
the uniform distribution in $B$, i.~e.~are distributed according to the
probability measure $\frac{\lambda\lfloor B}{\lambda(B)}$. When $B$ exhausts the entire
Euclidean space, the distribution of the point cloud $\{X\}$ converges to the
Poisson point process\footnote{We sometimes address our ensemble $\{X\}$ as
``canonical ensemble'' in view of its fixed particle number $N$, 
as opposed to the ``grand-canonical ensemble'' given by the Poisson point process,
with its Poisson-distributed particle number.}, 
of unit intensity as a consequence of $N=\lambda(B)$. 
Let us also specify to $B=B_L$, the ball of radius $L$ (centered at the origin). 

\medskip

After these specifications, the optimal matching problem comes with two input length scales, 
a macroscopic and a microscopic one: 
The macroscopic scale is the typical scale of the domain $B$, here the radius $L$.
The microscopic scale is the typical distance between the points $X$,
which here is (of order) unity, in view of the unit intensity.
It also comes with an output length scale, namely the optimal transportation distance $R$.
Clearly, we expect $1\ll R\ll L$. 
The following dimensional analysis of (\ref{ao21}) shows that the case of $d=2$ is critical:
The typical size of the l.~h.~s.~of (\ref{ao21}), 
as given by its standard deviation\footnote{the $\approx$
holds for $\lambda(\Sigma)\ll\lambda(B)$ and turns into a $=$ for the Poisson point process}
\begin{align*}
\sqrt{\mathbb{E}((\mu-\lambda)(\Sigma))^2}
=\sqrt{\mathbb{E}(\mu(\Sigma)-\mathbb{E}\mu(\Sigma))^2}
\approx\sqrt{\lambda(\Sigma)},
\end{align*}
has units of $\mbox{length}^{d/2}$. In view of (\ref{ao22}), the r.~h.~s.~of
(\ref{ao21}) has units of $(\mbox{units of $R$})\times\mbox{length}^{d-1}$. 
Hence if and only if $d=2$, these heuristics suggest that $R$ is dimensionless
\footnote{it also correctly predicts that $R$ has units of $\mbox{length}^{1/2}$ iff $d=1$}. 

\medskip

Confirming these heuristics, Leighton and Shor prove in the case of $d=2$,
to which we will henceforth restrict

\begin{theorem}[{Leighton-Shor \cite[Theorem 2]{LS}}]\label{T:5}
\begin{align}\label{ao23}
R\sim\ln^{3/4} L\quad\mbox{for}\quad L\gg 1\quad\mbox{with overwhelming probability}.
\end{align} 
\end{theorem}

Note that the argument of the logarithm in (\ref{ao23}) ought to be non-dimensional and 
thus is to be interpreted as the {\it ratio} of the macroscopic and the microscopic input length.
In the sequel, we will elucidate the seemingly mysterious exponent $3/4$, 
which cannot be inferred from a dimensional argument.

\subsection{An isoperimetric problem as a dual of the optimal matching problem}\label{ss:dual}

Our goal is to go beyond the scaling in (\ref{ao23}) and to identify the
prefactor. For this, we turn to the dual problem of maximization over subsets $\Sigma\subset B$.
Two approximations seem natural:
\begin{itemize}
\item[1)] We approximate the Poissonian shot noise by Gaussian white noise, 
which amounts $(\mu-\lambda)(\Sigma)$ 
$\approx W(\Sigma)$ $=\int_\Sigma\xi$, where $\xi$ is white noise in two dimensions.
This is a good approximation on scales $\gg 1$ and $\ll L$. In order to avoid the latter
constraint, we have to use that $(\mu-\lambda)(\Sigma)$ is better approximated by
$W(\Sigma,B)$ $=\int_\Sigma\xi-\frac{\lambda(\Sigma)}{\lambda(B)}\int_B\xi$,
which retains the symmetry property 
\begin{align}\label{ao28}
W(B-\Sigma,B)=-W(\Sigma,B).
\end{align}
However, this Gaussian approximation removes the small-scale cut-off provided 
by the typical distance between points (which we recall is unity).
Hence the approximation of the functional 
has to be complemented by a modification of the configuration space 
$\{\Sigma\subset B\}$, into which the small-scale cut-off has now to be incorporated.
We do this by restricting to polygons $\Sigma$ of edge length $\ge 1$,
which is somewhat arbitrary but convenient.
\item[2)] We take (\ref{ao22}) as an inspiration to replace the r.~h.~s.~of
(\ref{ao21}) by $R\,P(\Sigma\cap B)$. In a classical setting, this is known to be
a good approximation if $R$ is small compared to the radius of curvature of $\Sigma$.
In our polygonal setting, this is a good approximation when $\Sigma$ is flat on scales
of order $R$. We say that $\Sigma$ is flat in some ball $B'$ if
$\Sigma\cap B'$ is given by the sub-level set of a function with small Lipschitz
constant. 
\end{itemize}

\medskip

After these approximations, the dual problem simplifies to
\begin{align}\label{ao24}
I(B)=\sup\big\{\,\frac{W(\Sigma,B)}{P(\Sigma,B)}
\,\big|\,\Sigma\subset B\;
\mbox{polygon of side length $\ge 1$}\,\big\}.
\end{align}
The previous steps have indeed been made rigorous in the upcoming work \cite{COPW}. In view of the characterizing property of white noise we find for
the standard deviation of the numerator in (\ref{ao24}) that
$\sqrt{\mathbb{E}(W(\Sigma,B))^2}$ $=\sqrt{\lambda(\Sigma)(1-\frac{\lambda(\Sigma)}{\lambda(B)})}$ so that the maximization problem
in (\ref{ao24}) has the flavor of the classical isoperimetric problem (in two space dimensions),
motivating the letter $I$. 
In this setting the result of Leighton-Shor takes the form
\begin{align}\label{ao25}
I(B_L)\sim\ln^{3/4}L\quad\mbox{with overwhelming probability}.
\end{align}
We note for later reference that for fixed but arbitrary $\Sigma$,
$\frac{W(\Sigma,B)}{P(\Sigma,B)}$ is Gaussian random variable
of variance
$\frac{\sqrt{\lambda(\Sigma)(1-\frac{\lambda(\Sigma)}{\lambda(B)})}}{P(\Sigma,B)}\leq\frac{\sqrt{\min\{\lambda(\Sigma)
,\lambda(B-\Sigma)\}}}{P(\Sigma,B)}$, 
which by the classical isoperimetric
inequality is $O(1)$. Hence by Gaussian concentration, which transmits to the supremum of
random variables, we have
\begin{align}\label{ao26}
|I(B)-\mathbb{E}I(B)|\lesssim 1\quad\mbox{with overwhelming probability}.
\end{align}
In particular, we learn that (\ref{ao25}) is equivalent to the same statement just
on the level of the expectation, i.~e.~$\mathbb{E}I(B_L)\sim\ln^{3/4}L$.

\medskip

Let us revisit the Approximations 1 and 2; for that purpose, we consider
a maximizer $\Sigma^*$ in (\ref{ao24}) for $B=B_L$.
Approximation 1 is legitimate if the micro scale $1$ is much smaller 
than the scale on which $\Sigma^*$ is flat.
Approximation 2 is legitimate if $\Sigma^*$ is flat on scale $\ln^{3/4}L$.
Hence the approximations are self-consistent if $\Sigma^*$ is flat
up to scales that are large w.~r.~t.~$\ln^{3/4}L$. 
Hence it is natural to explore the regularity of $\Sigma^*$, which we will do next.


\subsection{Ried-Wagner on regularity, applied to the isoperimetric problem} 
Ried and Wagner \cite[Theorem 1.1]{RW} (rigorously) derived a regularity result for a minimizer $\Sigma^*$ 
of (\ref{ao24}) with $B=B_L$ which informally justifies the Approximation 1 and 2 
of Subsection \ref{SS:LS}.
We present their result in a way that is targeted to this justification,
and in a slightly strengthened version following \cite{COP}. For the discussion, it is convenient to
slightly modify (\ref{ao24}) by replacing the relative perimeter $P(\Sigma,B)$ within $B$
by the absolute one $P(\Sigma)$ and to return to $W(\Sigma)=\int_{\Sigma}\xi$,
while mimicking the symmetry (\ref{ao28}) in maximizing the absolute value $\frac{|W|}{P}$.
After this modification, in view of the elementary inequality $\frac{W\pm W'}{P+P'}$ 
$\le\max\{\frac{|W|}{P},\frac{|W'|}{P'}\}$, the maximization automatically extends to 
sums and differences of $\Sigma$'s. It is thus convenient to express the boundaries of
such sums and differences by oriented curves (that are neither simply connected nor 
connected); then $P(\Sigma)$ corresponds to their total length, and $W(\Sigma)$ 
to the integral of $\xi$ against the winding number of the curve:
\begin{align}\label{ao29}
I(B)=\sup\big\{\,\frac{W(\Sigma)}{P(\Sigma)}
\,\big|\,\Sigma\subset B\;
\mbox{polygonal curve of edge length $\ge 1$}\,\big\}.
\end{align}
This point of view\footnote{which in geometric analysis would be called currents} 
was first explicitly adopted in \cite[Subsection 5.1]{DWARXIV}.
In fact, we first note that the maximizer $\Sigma^*$ of (\ref{ao29}) 
is easily seen to be a maximizer of the additive problem
\begin{align}\label{ao27}
\epsilon W(\Sigma)-P(\Sigma)\quad\mbox{where}\quad\epsilon:=\frac{1}{I(B_L)},
\end{align}
which was the starting point in \cite{RW}.

\medskip

The first ingredient of \cite[Lemma 5.1]{RW} is another ``algebraic'' observation that locally allows
to interpret $\epsilon W(\Sigma)$ in (\ref{ao27}) as a perturbation: If $\Sigma$ coincides with
$\Sigma^*$ outside of some ball $B'\subset B_L$, we have by definition (\ref{ao29}) 
which allows us to think of $\Sigma^*-\Sigma$ as described by an oriented curve 
of total length given by the perimeter $P(\Sigma\triangle\Sigma^*)$ of the symmetric difference:
\begin{align*}
W(\Sigma^*)-W(\Sigma)
\le I(B')P(\Sigma^*\triangle\Sigma)
\le I(B')\big(P(\Sigma^*\cap B')+P(\Sigma\cap B')\big).
\end{align*}
Hence any minimizer $\Sigma^*$ of (\ref{ao27}) 
is a quasi-minimizer of the minimal (one-dimensional) area problem:
For any ball $B'$ we have
\begin{align}\label{ao03}
\lefteqn{P(\Sigma^*\cap B')\le\frac{1+\theta(B')}{1-\theta(B')}P(\Sigma\cap B')}\nonumber\\
&\quad\mbox{provided}\quad\Sigma^*\triangle\Sigma\subset B'
\quad\mbox{and}\quad\theta(B'):=\frac{I(B')}{I(B_L)}<1.
\end{align}

\medskip

The second ingredient is of stochastic nature (as before with the implicit understanding
that the statements hold with overwhelming probability): We start from (\ref{ao25}),
which by stationarity of $\xi$ and thus shift-invariance in law of (\ref{ao24}) holds 
no matter what the center of the ball is:
\begin{align}\label{ao70}
I(B)\sim\ln^{3/4}l\quad\mbox{where $l\gg 1$ is the radius of $B$}.
\end{align}
Hence in order to control
the supremum over a class of balls, it suffices to control the 
fluctuations $I(B)-\mathbb{E}I(B)$. By monotonicity of $I$ in $B$,
it is enough to consider $O((L/l)^2)$ many balls $B$ of radius $l$ that cover $B_L$. Hence by Gaussian concentration, 
recall the discussion around (\ref{ao26}),
taking the supremum comes with the additive cost of the square root of the logarithm of this number:
\begin{align*}
\max_{B\subset B_L\;\text{of radius $\le l$}} I(B)\lesssim
\ln^{3/4}l+\ln^{1/2}(L/l).
\end{align*}
Hence we obtain for the $\theta$ defined in (\ref{ao03})
\begin{align*}
\max_{B'\subset B_L\;\text{of radius $\le l$}}\theta (B')\lesssim(\frac{\ln l}{\ln L})^{3/4}
+(\frac{1}{\ln L})^{1/4},
\end{align*}
where the last term is dominant for sufficiently small $l$, as relevant below.
%
%

\medskip

The last ingredient is of deterministic and geometric nature,
namely a regularity theory for quasi-minimizer of parameter $\theta$
in all balls $B'\subset B$ of radius $\le l$:
\begin{align}\label{ao04}
P(\Sigma^*\cap B')\le\frac{1+\theta}{1-\theta}P(\Sigma\cap B')
\quad\mbox{provided}\;\Sigma^*\triangle\Sigma\subset B'\nonumber\\
\stackrel{\sqrt{\theta}\ln l\ll 1}{\Longrightarrow}\quad 
\mbox{oscillation of the normal of $\Sigma^*$ in $B$}\;\lesssim\;\sqrt{\theta}\,\ln l,
\end{align}
which in particular means that $\Sigma^*$ is flat in $B$.
Statement (\ref{ao04}) arises from the combination of an elementary geometric inequality
obtained from comparing $\partial\Sigma\cap B'$ to a line segment, and that can be brought
into the form of an estimate
of the (squared) $L^2$-oscillation\footnote{called the excess} of the normal,
namely $\inf_{\bar\nu\,\mbox{\tiny const}}\int_{\partial\Sigma^*\cap B'}|\nu^*-\bar\nu|^2$ 
$\lesssim\theta P(\Sigma^*\cap B')$, and an iteration\footnote{called a Campanato iteration} over 
the $\ln l$ dyadic scales from $l$ down to the micro scale unity on which the normal is constant.

\medskip

The three ingredients combine to flatness on scales that are sub-algebraically small
w.~r.~t.~$L$ 
%
\begin{align*}
&\ln l\ll\ln^{1/8} L\quad\Longrightarrow\quad\Sigma^*\;\mbox{is flat on scales}\;l.
\end{align*}
%
\ignore{
{\color{blue}
With the blue estimate above
\begin{align*}
&\theta\lesssim(\frac{\ln l}{\ln L})^{3/4}+\frac{1}{\ln^{1/4}L}\\
&\mbox{hence}\qquad\theta\ll1\qquad\mbox{in the regime}\qquad\ln l\ll\ln L, \ln L\gg1.
\end{align*}
The regime is the optimal one (but the quantitative control on $\theta$ is not).
When converting almost-minimality of the perimeter into flatness:
\begin{align*}
\sqrt{\theta}\ln l\ll1\qquad\mbox{implied by}\qquad\ln l\ll\ln^{1/8}L.
\end{align*}
}
}
Obviously, this upper bound on $l$ is still much larger than $\ln^{3/4} L$, which 
shows self-consistency of the Approximations 1) and 2) in Subsection \ref{SS:LS}.


\subsection{Cheng-Otto-Palmieri on the geometric linearization}\label{SS:COP}

Given a configuration $\Sigma$ and an (integer) length $1\ll l\ll L$, 
let the coarse-grained configuration $\Sigma_{\ge l}$ denote the polygon 
that arises from only keeping every $l$-th vertex. We write
\begin{align}\label{ao05}
\left\{\begin{array}{rcl}
P(\Sigma)&=&P(\Sigma_{\ge l})+\big(P(\Sigma)-P(\Sigma_{\ge l})\big)\quad\mbox{and}\\[1ex]
W(\Sigma)&=&W(\Sigma_{\ge l})+\big(W(\Sigma)-W(\Sigma_{\ge l})\big).
\end{array}\right.
\end{align}
Under the assumption that the second summands are much smaller than the
first ones in (\ref{ao05}), we obtain for every exponent $\alpha>0$ to be chosen soon that
%
%
%
\begin{align*}
\lefteqn{\frac{1}{\alpha}(\frac{W(\Sigma)}{P(\Sigma)})^\alpha
\approx\frac{1}{\alpha}(\frac{W(\Sigma_{\ge l})}{P(\Sigma_{\ge l})})^\alpha
+(\frac{W(\Sigma_{\ge l})}{P(\Sigma_{\ge l})})^\alpha}\nonumber\\
&\times\frac{1}{P(\Sigma_{\ge l})}
\Big(\frac{P(\Sigma_{\ge l})}{W(\Sigma_{\ge l})}
\big(W(\Sigma)-W(\Sigma_{\ge l})\big)-\big(P(\Sigma)-P(\Sigma_{\ge l})\big)\Big).
\end{align*}

\medskip

Under the assumption that
\begin{align}\label{ao11}
\Sigma\;\mbox{is flat on scales}\;l,
\end{align}
the edges of $\Sigma_{\ge l}$ have length $\gtrapprox l$.
The second parts in (\ref{ao05}) can be naturally written as the sum over
$N\approx\frac{P(\Sigma_{\ge l})}{l}$ summands we index by $n=1,\cdots,N$
and each involving an open polygonal curve
$\Sigma_n$ closed by a straight line segment $\Sigma_{\ge l,n}$ of length $l_n$
%
%
also reversed here
\begin{align}\label{ao09}
\lefteqn{\frac{1}{\alpha}(\frac{W(\Sigma)}{P(\Sigma)})^\alpha
\approx\frac{1}{\alpha}(\frac{W(\Sigma_{\ge l})}{P(\Sigma_{\ge l})})^\alpha}\nonumber\\
&+\epsilon^{-\alpha}\big(\sum_{n=1}^Nl_n\big)^{-1}\sum_{n=1}^N
\Big(\epsilon
\big(W(\Sigma_n)-\big(P(\Sigma_n)-l_n\big)\Big)
\end{align}
%
where $\epsilon:=\frac{P(\Sigma_{\ge l})}{W(\Sigma_{\ge l})}$.
Still under assumption (\ref{ao11}), we may write
\begin{align}\label{ao10}
\partial\Sigma_n=\mbox{graph of}\;h_n\;\mbox{over the line segment}\;\Sigma_{\ge l,n},
\end{align}
%
so that in the corresponding coordinate system,
\begin{align*}
&\epsilon
W(\Sigma_n)-\big(P(\Sigma_n)-l_n\big)=\epsilon W_n(h_n)-\int_0^{l_n}dx\big(\sqrt{1+(\frac{dh_n}{dx})^2}-1\big)\nonumber\\
&\mbox{with}\quad W_n(h):=\int_0^{l_n}dx\int_0^{h(x)}dy\xi_n(x,y),
\end{align*}
where by the invariance in law of white noise $\xi$ under
translation (stationarity) and rotation (isotropy), provided $\Sigma_{\geq l}$ is deterministic, $\xi_n$ is a new instance of white noise.

\medskip

Since (\ref{ao11}) means $|\frac{dh_n}{dx}|\ll 1$, we may expand the square root to the effect of
\begin{align}\label{ao08}
\epsilon W(\Sigma_n)-\big(P(\Sigma_n)-l_n\big)
\approx\epsilon W_n(h_n)-D(h_n),
\end{align}
where $W_n=_{\text{law}}W$ and $D$ are the functionals defined in Subsection \ref{SS:setting} (with $l_n$ replacing $L$).
%
%
Approximating the area functional $P$ by the Dirichlet integral $D$, an instance of
a harmonic approximation, amounts to a geometric linearization, 
and is well-known from the theory of minimal surfaces.
We note that under the rescaling
\begin{align}\label{ao07}
h=\epsilon^{2/3}\hat h\;\;\mbox{we have}\;\;
D(h)=\epsilon^{4/3}D(\hat h),\;\;W(h)=\epsilon^{1/3}\hat W(\hat h),
\end{align}
where by the behavior of the law of white noise under the (anisotropic)
linear transformation $(x,y)=(x,\epsilon\hat y)$, $\hat W$ has the
same law as $W$.
Hence inserting (\ref{ao07}) into (\ref{ao08}) and then (\ref{ao09}) we learn 
that for $\alpha=4/3$
\begin{align}\label{ao10bis}
\frac{3}{4}(\frac{W(\Sigma)}{P(\Sigma)})^{4/3}
\approx\frac{3}{4}(\frac{W(\Sigma_{\ge l})}{P(\Sigma_{\ge l})})^{4/3}
+\big(\sum_{n=1}^Nl_n\big)^{-1}\sum_{n=1}^N(\hat W_n(\hat h_n)-D(\hat h_n)),
\end{align}
under the relations (\ref{ao10}) and (\ref{ao07}), and the assumption (\ref{ao11}).

\medskip

It turns out that the l.~h.~s.~of (\ref{ao10bis}) is approximately maximized
by maximizing the summands on the r.~h.~s.~separately, which in particular means that
coarse-graining and maximization approximately commute
-- in our variational context this is derived by establishing the corresponding upper
and lower bounds separately, see Subsection \ref{SS:cutandpaste} for our geometrically
linear setting.

\medskip

This provides a motivation to consider (\ref{ao15}) next to $I_L:=I(B_L)$, cf.~(\ref{ao29}).
%
%
%
%
Note that in view of (\ref{ao11}) we have that the edge lengths $l_n$
of $\Sigma_{\ge l}$ are $\gtrapprox l$. 
Hence assuming that the coarse-grained
maximizer is close to achieving a maximum of the coarse-grained problem, i.~e.~close to
maximizing $W(\Sigma)/P(\Sigma)$ over all polygonal $\Sigma\subset B_L$ with  
edge length $\ge l$, we obtain from the scaling invariance in law of $W/P$
that $W(\Sigma_{\ge l}^*)/P(\Sigma_{\ge l}^*)$ $\approx_{\rm law} I_{L/l}$.
Likewise assuming that $\hat h_n^*$ are close to maximizing $\hat W_n(\hat h_n)-D(\hat h_n)$,
we obtain $\hat W_n(\hat h_n^*)-D(\hat h_n^*)$ $\approx_{\rm law}l_n^*A_{l}$. 
In connection with the strong concentration of $I_L$, cf.~(\ref{ao26}), 
this suggests the functional relation
\begin{align}\label{ao16}
{\textstyle\frac{3}{4}}(\mathbb{E}I_{L})^{4/3}\approx
{\textstyle\frac{3}{4}}(\mathbb{E}I_{L/l})^{4/3}+\mathbb{E}A_{l},
\end{align}
which indeed is an intermediate statement leading to Theorem \ref{T:4} below.


\subsection{Otto-Palmieri-Wagner on the geometrically linear problem}\label{SS:OPW}
The coarse-graining strategy of Subsection \ref{SS:COP} also applies to 
the geometrically linearized problem (\ref{ao15}). In this case, the analogue
$h_{\ge l}$ of $\Sigma_{\ge l}$ is even simpler: It is the piecewise linear interpolation
of $h$ on intervals of length $l$ introduced in Subsection \ref{SS:refine}. 
In this geometrically linear setting, the analogue of (\ref{ao10bis}) is exact:
\begin{align}\label{ao32}
(W-D)(h)/L=(W-D)(h_{\ge l})/L+N^{-1}\sum_{n=1}^N(W_n-D)(h_n)/l,
\end{align}
and now $h_n$ is just the restriction of $h-h_{\ge l}$ 
on the $n$-th subinterval of $[0,L]$ on size $l$, shifted to the reference interval $[0,l]$. 
A simpler version of the 
heuristic argument that lead to (\ref{ao16}) now suggests that we may expect
\begin{align}\label{ao18}
\mathbb{E}A_{L}\approx\mathbb{E}A_{L/l}+\mathbb{E}A_{l};
\end{align}
see Subsections \ref{SS:cutandpaste} and \ref{SS:tackle} below.

\medskip

Note that (\ref{ao18}) corresponds to an additivity of the action per volume $A_L$
in terms of the logarithm $\ln L$ of the ratio between macro and micro scale,
motivating Theorem \ref{T:2}.
This in in turn suggests that there exists an $a^*$ independent of $L$ such that 
\begin{align}\label{ao19}
\mathbb{E} A_L\approx a^*\ln L\quad\mbox{for}\;L\gg 1,
\end{align}
which is a consequence of Theorem \ref{T:1}.
Inserting (\ref{ao19}) into (\ref{ao16}) suggest that
$\frac{3}{4}(\mathbb{E}I_L)^{4/3}\approx a^*\ln L$, that is,
\begin{align}\label{ao20}
\mathbb{E}I_L\approx({\textstyle\frac{4}{3}}a^*\ln L)^{3/4}\quad\mbox{for}\;L\gg 1,
\end{align}
which is consistent with the scaling (\ref{ao25}) established in \cite[Theorem 2]{LS},
and determines the prefactor in terms of the simpler (geometrically linear) problem (\ref{ao15}).
In forthcoming work with Xiaopeng Cheng we establish (\ref{ao20}):

\begin{theorem}\label{T:4} With the $a^*$ from Theorem \ref{T:1} we have
\begin{align*}
\lim_{L\uparrow\infty}\mathbb{E}I_L/\ln^{3/4} L=({\textstyle\frac{4}{3}}a^*)^{3/4}.
\end{align*}
\end{theorem}

Up to the approximation in Subsection~\ref{ss:dual}, Theorem~\ref{T:4} identifies the leading-order constant in the optimal matching problem (Theorem~\ref{T:5}). This prompts to compare the status in $L^\infty$ matching, which is considered here, to the more standard $L^2$ matching. Also there the two-dimensional case is critical, but diverges as $\ln^{1/2}L$ \cite{AKT}. In the $L^2$ setting, a different type of geometric linearization was proposed in \cite{CLPS} and made rigorous in \cite[Theorem 1.1]{AST}. The presently best error estimate was derived in \cite[Theorem 1]{GHO}. 


\subsection{Talagrand on the maximum of a family of Gaussians} 

The problem of controlling $I_L$ from above can be tackled be monitoring the supremum of centered Gaussian variables
\begin{align}\label{m44}
\sup W(\Sigma)/\bar P\quad\mbox{over polygons $\Sigma$ with side-length $\ge1$ and $P(\Sigma)\le\bar P$},
\end{align}
 for every threshold $\bar P>0$ and then optimizing in $\bar P$. This restricted supremum can be treated by the general chaining theory developed by Talagrand \cite[Chapter 2]{T22}, which is a vast generalization of Kolmogorov's continuity argument. To this purpose, one endows the index set $\mathcal{S}(\bar P)$ of polygons in \eqref{m44} with the distance
\begin{align*}
d(\Sigma,\Sigma'):=\sqrt{\mathbb{E}\big(W(\Sigma)-W(\Sigma')\big)^2}=\sqrt{\lambda(\Sigma\Delta\Sigma')}.
\end{align*}
The supremum in \eqref{m44} is then bounded, with overwhelming probability and up to a constant, by a certain functional $\gamma_2(\mathcal{S}(\bar P),d)$, which appropriately measures the size of the metric space $(\mathcal{S}(\bar P),d)$. With this strategy, Talagrand reproved in an efficient way the scaling result \eqref{ao25} of Leighton and Shor (see~\cite[Section 4.7]{T22}).


\subsection{Aizenman-Wehr on random-field Ising model}
The following discrete version of the additive model (\ref{ao27}) is well-known:
$\mathbb{R}^2$ is replaced by the integer lattics $\mathbb{Z}^2$, 
(polygonal) subsets $\Sigma\subset\mathbb{R}^2$ are written in
terms of their characteristic functions $\sigma\colon\mathbb{R}^2\rightarrow\{0,1\}$
and interpreted as spin configurations, the perimeter is mimicked by the attractive
nearest-neighbor spin interaction from the Ising model
\begin{align*}
P(\sigma)=\sum_{z\sim z'}|\sigma(z)-\sigma(z')|,
\end{align*}
and $\xi$ now is white noise on $\mathbb{Z}^2$, i.~e.~a collection $\{\xi(z)\}_{z\in\mathbb{Z}^2}$
of independent standard Gaussian variables, so that
\begin{align*}
W(\sigma)=\sum_{z}\xi(z)\sigma(z)
\end{align*}
plays the role of a random external field term. Note that passing from 
(\ref{ao27}) to this model breaks the rotational symmetry in law.

\medskip

The question whether there is a phase transition despite the absence of thermal
noise has been debated and answered to the negative by Aizenman-Wehr \cite[Corollary 4.3 with $T=0$]{AW}:
Almost surely, there exists a unique minimizer $\sigma^*$ of the Hamiltonian
$P-\epsilon W$ in the sense that any
other configuration $\sigma$ that differs from $\sigma^*$ in only a finite number
of sites does not decrease the Hamiltonian. Note that the differences
$P(\sigma)-P(\sigma^*)$ and $W(\sigma)-W(\sigma^*)$ can be given a sense even if
$W(\sigma^*)$ and $W(\sigma^*)$ are infinite.
 


\subsection{Ding-Wirth on the divergence of the correlation length}
Consider the maximizer $\sigma^*$ of $\epsilon W-P$ on a finite sub-lattice
$B_L\cap\mathbb{Z}^2$ with zero boundary conditions.
The result of \cite{AW} suggests that as $L\uparrow\infty$, 
the influence of the boundary condition on 
expected spin at the center $\mathbb{E}\sigma^*(0)$ fades away, 
i.~e.~that $\lim_{L\uparrow\infty}\mathbb{E}\sigma^*(0)=\frac{1}{2}$.
Ding-Wirth showed that the increasing function $\mathbb{E}\sigma^*(0)$ of $L$
crosses the value $\frac{1}{3}$ when
\begin{align}\label{ao71}
\epsilon\ln^{3/4}L\sim 1.
\end{align}
As pointed out in \cite[Section 2.2]{DW}, the lower bound in (\ref{ao71}) follows easily from the upper bound
in (\ref{ao25}), for which they rely on the result by Talagrand \cite[Theorem 4.7.2]{T22}.
Indeed, on the one hand, if $\mathbb{E}\sigma^*(0)\ge\frac{1}{3}$, 
we must have $\sigma^*\not\equiv 0$ with a sizable probablity. 
On the other hand, if $\epsilon\ln^{3/4}L\ll 1$ we have 
by (\ref{ao25}) with overwhelming probability that for any configuration $\Sigma$,
$\epsilon W(\Sigma)-P(\Sigma)<0$ unless $\Sigma=\emptyset$.



\section{Strategy of proof}

\subsection{Sub- and super-additivity of the maximal action \texorpdfstring{$A$}{A}}\label{SS:cutandpaste}

A basic but crucial observation is that the maximum of $(W-D)/L$ over the
\begin{align}\label{ao56}
&\mbox{space of coarse-grained configurations}\nonumber\\
&:=\big\{\,\bar h\colon[0,L]\rightarrow\mathbb{R}\;
\mbox{continuous and linear on $[(n-1)l,nl]$ for all $n$}\,\big\},
\end{align}
when restricted to null Dirichlet boundary conditions, has the same law as $A_{L/l}$.
As a consequence, for a maximizer $h^*$ of $A_L$, $(W-D)(h^*_{\ge l})/L$ 
is not larger in law than $A_{L/l}$, 
see \eqref{m43} for the definition of $h_{\ge l}$.
We now recall the splitting (\ref{ao32}) and its notation
\begin{align}
h_n&=(h-h_{\ge l})(\cdot+(n-1)l)\label{ao57},\\
W_n(h_n)&=W_{[(n-1)l,nl]}\big(h_{\ge l}+h_n(\cdot-(n-1)l)\big)-W_{[(n-1)l,nl]}\big(h_{\ge l}\big),
\label{ao47}
\end{align}
where $W_I$ indicates that the definition (\ref{e71}) involves the $x$-interval $I$.

\medskip

We apply (\ref{ao32}) to the maximizer $h^*$ of $A_L$, 
so that by the previous observation 
the first r.~h.~s.~term in (\ref{ao32}) is $\le_{\rm law} A_{L/l}$.
Note that by definition, $h_n$ has homogeneous Dirichlet boundary data on $[0,l]$,
so that it is admissible in $A_l$. Note also that as a consequence of translation and shear invariance of white noise,
\begin{align}\label{m33}
(x,y)=(\hat x,\hat y+\phi(x))\qquad\mbox{implies}\quad\xi=_{\text{law}}\hat\xi,
\end{align}
applied to $\phi=h_{\ge l}$, the functional $(W_n-D)/l$ has the same law as $(W_{[0,l]}-D)/l$.
Momentarily ignoring the fact that the latter is wrong if we replace the deterministic
$h$ by the random $h^*$ we
would recover sub-additivity in the weaker form
\begin{align}\label{ao45}
&A_L\le A_{L/l,0}+(L/l)^{-1}\sum_{n=1}^{L/l}A_{l,n}\nonumber\\
&\mbox{with}\quad A_{L/l,0}=_{\rm law}A_{L/l}
\quad\mbox{and, naively,}\quad A_{l,n}=_{\rm law}A_l.
\end{align}
We thus would have in particular
\begin{align}\label{ao46}
\mbox{naively}\quad 
\mathbb{E}A_L\le\mathbb{E}A_{L/l}+\mathbb{E}A_{l},
\end{align}
which is the sub-additive version of the additive (\ref{ao18}).

\medskip

To actually recover (\ref{ao18}), 
 we apply the splitting (\ref{ao32}) again, this time with the
maximizer $\bar h^*$ of $A_L$ when restricted to (\ref{ao56})
playing the role of $h_{\ge l}$, 
and the maximizers $h_n^*$ of $A_{l,n}$ playing the role of $h_n$, one would hope for 
\begin{align}\label{ao39}
&A_L\ge A_{L/l,0}+(L/l)^{-1}\sum_{n=1}^{L/l}A_{l,n}\nonumber\\
&\mbox{with}\quad A_{L/l,0}=_{\rm law}A_{L/l}
\quad\mbox{and, naively,}\quad A_{l,n}=_{\rm law}A_l.
\end{align}


\subsection{Tackling statistical dependencies by introducing bins}\label{SS:tackle}~Let us explain more carefully why the splitting (\ref{ao32}) does not yield (\ref{ao45}).
Recall that the realization $W_n$ of the functional depends on $h_{\ge l}$ via (\ref{ao47}).
%
%
As long as $h_{\ge l}$ is deterministic, the functionals $W_n$ all have the law of $W$ and consequently,
\begin{align}\label{m34}
\mbox{for every deterministic $h_{\ge l}$}\qquad
A_{l,n}=A_{l,n}(h_{\ge l})=_{\text{law}}A_l,
\end{align}
where the random variables $A_{l,n}$ are those appearing in either \eqref{ao45} or \eqref{ao39}.
However, in order to deduce (\ref{ao45}), we apply (\ref{ao32}) to the random $h^*$
and thus random $h_{\ge l}^*$, so that the last statement in (\ref{ao45}) is not correct.

\medskip

In order to retain the identical law in \eqref{m34}, a first idea is to take the maximum over
the space the realization $h_{\ge l}^*$ lives in, namely (\ref{ao56}),
\begin{align}
\sup\{A_{l,n}(\bar h)~|~\bar h~\mbox{coarse grained configuration}~\eqref{ao56}\}\label{m28}
\end{align}
%
%
%
or rather a bounded subset thereof
%
%
defined by suitable bounds we expect $h^*_{\ge l}$ to satisfy with overwhelming probability,
see Subsection \ref{SS:bounds}.

\medskip

Rather than taking
a supremum over such an unwieldy continuum set 
a better idea is to pass to a discrete set, 
defined by constraining the functions to have values in $l\mathbb{Z}$ at the grid points $ln$:
\begin{align}\label{ao48}
\lefteqn{
\mbox{net of coarse-grained configurations}\;{\mathcal N}}\nonumber\\
&:=\big\{\,\bar h\;\mbox{coarse-grained configuration}\,\big|\,
\bar h(nl)\in l\mathbb{Z}\;\mbox{for all $n$}\,\big\}.
\end{align}
For a given coarse-grained configuration $\bar h$,
we denote by $\Pi(\bar h)$ the closest configuration in the net ${\mathcal N}$ as determined by
\begin{align}\label{ao62}
\bar h(nl) \in \big[\Pi(\bar h)(nl)-l/2 , \Pi(\bar h)(nl)+l/2 \big)\nonumber\\
\mbox{so that}\quad|(\bar h-\Pi(\bar h))(nl)|\le l/2;
\end{align}
this projection clearly bins the space of coarse-grained configurations.

\medskip

The idea is to take the supremum \eqref{m28} in two steps: For every configuration $\bar h\in\mathcal{N}$, we first introduce the supremum over all coarse-grained configurations with the same projection $\bar h$,
\begin{align}
&A_{l,n}^+(\bar h):=\sup\{A_{l,n}(\bar h')~|~\mbox{$\bar h'$ with $\Pi(\bar h')=\bar h$}\},\nonumber\\
&\mbox{so that}\quad A_{l,n}(h_{\ge l}^*)\le A_{l,n}^+(\Pi(h_{\ge l}^*)).\label{m16}
\end{align}
The advantage is that we now can take the supremum over the discrete set of objects $\Pi(h^*_{\ge l})\in\mathcal{N}$ (instead of the continuum $h^*_{\ge l}$), at the expense of having to deal with the (larger) action $A_{l,n}^+$. 

\medskip

Since \eqref{m34} holds in the stronger form
\begin{align*}
A_{l,n}(\cdot+\bar h)=_{\text{law}}A_{l,n}(\cdot)\quad\mbox{(stationarity)},
\end{align*}
also the
variables $\{A_{l,n}^+(\bar h)\}_{\bar h, n}$ have all the same law. Hence it suffices to give a more explicit description of the particular one with $\bar h=0$ and $n=1$. Using the orthogonality relation
\begin{align*}
&D(h)=D(a_h)+D(h-a_h)\\
&\mbox{where}\quad a_h:=\mbox{linear function such that $a_h(0)=h(0),a_h(l)=h(l),$}
\end{align*}
and recalling the definition \eqref{ao47} of $W_n$ for $n=1$ allows to write $ A_{l,1}^+(0)$ as 
\begin{align}\label{m38}
 A_l^+:=&\max_{|y_0|,|y_1|\le l/2}\max\big\{\,(W-D)(h)/l-(W-D)(a_h)/l\,\big|\,\nonumber\\
&h\colon[0,l]\rightarrow\mathbb{R}~\mbox{is continuous and linear on
$[n-1,n]$ for $n=1,\cdots,l$}\nonumber\\&\mbox{with}\,\,h(0)=y_0,h(l)=y_1\big\}.
\end{align}

\medskip

With this notation at hand, let us now make the strategy in \eqref{ao45} rigorous at the level of the variables $ A^+$. Consider the linear $a_h$, which now interpolates the boundary conditions of $h$ on the macroscopic interval $[0,L]$, and subtract its action from both sides on \eqref{ao32},
\begin{align}\label{m30}
&(W-D)(h)/L-(W-D)(a_h)/L\nonumber\\&=(W-D)(h_{\ge l})/L-(W-D)(a_h)/L+N^{-}1\sum_{n=1}^N(W_n-D)(h_n)/l.
\end{align}
Evaluating at the maximizer $h^*$ of $ A_L^+$, the difference on the r.~h.~s.~is $\le_{\text{law}} A^+_{L/l}$, so that the above reasoning yields the following:

\begin{lemma}\label{L:3}
\begin{align}\label{ao35bis}
 A^+_L\le A^+_{L/l,0}+(L/l)^{-1}\sum_{n=1}^{L/l} A_{l,n}^+(\Pi(h^*_{\ge l}))
\end{align}
where as before $ A_{L/l,0}^{+}=_{\rm law} A_{L/l}^+$. 
\end{lemma}

We employ a similar strategy for the super-additivity statement \eqref{ao39},
which also relies on the splitting \eqref{ao32} and is captured in terms of
\begin{align}\label{m39}
 A_l^-:=&\min_{|y_0|,|y_1|\le l/2}\max\big\{\,(W-D)(h)/l-(W-D)(a_h)/l\,\big|\,\nonumber\\
&h\colon[0,l]\rightarrow\mathbb{R}~\mbox{is continuous and linear on
$[n-1,n]$ for $n=1,\cdots,l$}\nonumber\\&\mbox{with}\,\,h(0)=y_0,h(l)=y_1\big\}.
\end{align}
For every couple of boundary conditions $y_0,y_1$, a competitor for the inner problem is constructed like in Subsection~\ref{SS:cutandpaste} by first picking the maximizer $\bar h^*$
over coarse-grained configurations \eqref{ao56} and then pasting $h_n^*$
optimizing $W_n-D$ with homogeneous Dirichlet boundary conditions. Analogously to \eqref{m16},
\begin{align*}
A_{l,n}(\bar h^*)\ge A_{l,n}^-(\Pi(\bar h^*))
\end{align*}
where the random variables $ A^-_{l,n}(\bar h)$ are defined by
\begin{align}
 &A_{l,n}^-(\bar h):=\inf\{A_{l,n}(\bar h')~|~\mbox{$\bar h'$ with $\Pi(\bar h')=\bar h$}\},\nonumber\\
&\mbox{and satisfy}\quad  A_{l,n}^-(\bar h)=_{\text{law}} A^-_l.\label{m17}
\end{align}
This two-scale construction of a competitor yields:

\begin{lemma}\label{L:3-}
\begin{align}\label{ao50}
 A^-_L\ge A^-_{L/l,0}+(L/l)^{-1}\sum_{n=1}^{L/l} A_{l,n}^-(\Pi(\bar h^*))
\end{align}
where as before $ A^-_{L/l,0}=_{\rm law} A^-_{L/l}$.
\end{lemma}

\subsection{Connecting the sub-additive \texorpdfstring{$A^+$}{A+} to the super-additive
\texorpdfstring{$A^{-}$}{A-}}\label{SS:connect}

In Subsections~\ref{SS:count} and \ref{SS:bounds}, we will deduce from Lemmas~\ref{L:3} and~\ref{L:3-}
\begin{align}
\mathbb{E} A^+_L&\le \mathbb{E} A^+_{L/l}+\mathbb{E} A_{l}^++O(1),\label{ao33}\\
\mathbb{E} A^-_L&\ge \mathbb{E} A^-_{L/l}+\mathbb{E} A_{l}^-+O(1),\label{ao51}
\end{align}
 cf.~Lemma~\ref{L:subsuper}, which will imply the convergence of $\mathbb{E} A^\pm_L/\ln L$. Since by definitions \eqref{m38} \& \eqref{m39}
\begin{align}\label{m26}
 A^-_L\le A_L\le A^+_L,
\end{align}
in order to get a limit for $A$, we need to control $ A^+$ from above by $ A^-$, as achieved in Lemma~\ref{L:pm} below. The strategy of monitoring increasing and decreasing quantities that sandwich the one of interest like in \eqref{m26} is a common strategy. The monotonicity of $A^+$ and $A^-$ typically arises from pasting and cutting arguments; they have been used to establish extensivity of the optimal energy in pattern formation \cite{ACO,OV} and in quantitative stochastic homogenization of convex functionals \cite{AS}.

\medskip

To this end, we first observe that the contribution of the linear part in the definitions \eqref{m38} \& \eqref{m39} of $A_L^\pm$ is negligible on the logarithmic scale,
\begin{align}
&\big\|\max_{|y_0|,|y_1|\le L/2}|W-D|(a_{y_0,y_1})/L\big\|_2\lesssim1,\quad\mbox{where}\nonumber\\
&a_{y_0,y_1}:=\mbox{linear function such that $a_{y_0,y_1}(0)=y_0,a_{y_0,y_1}(L)=y_1$}.\label{m40}
\end{align}
Indeed by scaling, the law of $\max_{|y_0|,|y_1|\le L/2}|W-D|(a_{y_0,y_1})/L$ does not depend on $L$, so that it is enough to show the finiteness of its $\|\cdot\|_2$ norm. This in turn follows from standard chaining arguments; for the convenience of the reader, we give an alternative proof in the appendix.

\medskip

The strategy now is the following:
One starts from the $A^+$-maximizer on $[L/4,3L/4]$, denoting by $y_0^*$ and $y_1^*$ its random
boundary data, which satisfy $|y_0^*|,|y_1^*|\le L/4$. Given $(y_0,y_1)$ with $|y_0|,|y_1|\le L/2$, one considers the two maximizers of $W-D$ on $[0,L/4]$ 
and $[3L/4,L]$ with boundary data $(y_0,y_0^*)$ and $(y_1^*,y_1)$, respectively. Concatenation yields a competitor for $A_L^-$. Note that $|y_0|\le L/2$ and $|y_0^*|\le L/4$, which is larger than $1/2$ times the interval size of $[0,L/4]$, which motivates the introduction of
\begin{align}
&\mbox{$\tilde A^-_l$ defined as $A^-_l$ in \eqref{m39}}\nonumber\\
&\mbox{but with $|y_0|,|y_1|\le l/2$ replaced by $|y_0|,|y_1|\le 2l$.}\label{m41}
\end{align}
Hence by definition \eqref{m38} of $A_L^+$ we obtain
\begin{align}
\lefteqn{A_L^{-}\ge\frac{1}{2}A_{L/2,0}^++\frac{1}{4}\tilde A_{L/4,1}^{-}
+\frac{1}{4}\tilde A_{L/4,2}^{-}+X}\label{m23}\\
&\mbox{with}\quad A_{L/2,0}^+=_{\rm law}A_{L/2}^+,\;\tilde A_{L/4,n}^-=_{\rm
law}\tilde A_{L/2}^-,\nonumber
\end{align}
where $X$ accounts for the action of the linear parts in the other four terms in \eqref{m23}, and which in view of \eqref{m40} is controlled by
\begin{align}\label{m42}
\|X\|_2\lesssim1.
\end{align}

\medskip

It remains to argue that the additional factor of 4 in the definition \eqref{m41} of $\tilde A^-_l$
is negligible on the logarithmic energy scale:
\begin{align}
\tilde A_L^{-}\ge_{\text{law}}A_L^-+O(1).\label{m24}
\end{align}
Indeed, by definition \eqref{m41} we have
\begin{align*}
\tilde A_L^-=\min\big\{ A^-_{L,\bar y_0,\bar
y_1}~\mbox{over all}~\bar y_0,\bar y_1\in L\{\textstyle{-\frac{3}{2},-\frac{1}{2},\frac{1}{2},\frac{3}{2}}\}\big\}
\end{align*}
where $ A_{L,\bar y_0,\bar y_1}^-$ is defined by changing the constraint in the definition \eqref{m39} of $ A_L^-$
\begin{align*}
&\mbox{from $|y_0|,|y_1|\le{\textstyle{\frac{L}{2}}}$ to $|y_0-\bar y_0|,|y_1-\bar y_1|\le{\textstyle{\frac{L}{2}}}$,}\\
&\mbox{and satisfies}\quad  A_{L,\bar y_0,\bar y_1}^-=_{\text{law}} A^-_L.
\end{align*}
Adding and subtracting $\mathbb{E} A_L^-$ and applying
$\|\sup_{i=1}^{16}X_i\|_1\lesssim\sup_{i=1}^{16}\|X_i\|_1$
to $X_i:= A_{L,\bar y_0,\bar y_1}^--\mathbb{E} A_L^-$, we deduce \eqref{m24} in the rigorous form of
\begin{align}\label{m32}
\mathbb{E}\tilde
A_L^--\mathbb{E} A_L^-\gtrsim-\| A_L^--\mathbb{E} A_L^-\|_1.
\end{align}
Taking the expectation of \eqref{m23} and inserting \eqref{m42} and \eqref{m32} yields:

\begin{lemma}\label{L:pm}
\begin{align}\label{ao41}
\mathbb{E}A_{L/2}^+-2\mathbb{E}A^-_L+\mathbb{E}A_{L/4}^{-}\lesssim\|A^-_{L/4}-\mathbb{E}A^-_{L/4}\|_1+1.
\end{align}
\end{lemma}

This is indeed the desired control of $A^+$ by $A^-$: Momentarily assuming
an $O(1)$ control of the r.~h.~s.~and
the existence of $\lim_{L\uparrow\infty}\mathbb{E}A^-_L/\ln L$, Lemma~\ref{L:pm} would
imply $\limsup_{L\uparrow\infty}\mathbb{E}A_L^+/\ln
L\le\lim_{L\uparrow\infty}\mathbb{E}A_L^-/\ln L$.



\subsection{Balancing cardinality and concentration}\label{SS:count}

The argument for (\ref{ao33}) and (\ref{ao51}) is based
on splitting (\ref{ao35bis}) and (\ref{ao50}) into expectation and fluctuation
\begin{align}
\mathbb{E} A^+_L&\le
\mathbb{E} A^+_{L/l}+\mathbb{E} A^+_l
+\mathbb{E}(L/l)^{-1}\sum_{n=1}^{L/l}\big( A_{l,n}^+(\Pi(\bar h^*_{\ge
l}))-\mathbb{E} A^+_l\big),\label{m7}\\
\mathbb{E} A^-_L&\ge\mathbb{E} A^-_{L/l}+\mathbb{E} A^-_l
+\mathbb{E}(L/l)^{-1}\sum_{n=1}^{L/l}\big( A_{l,n}^-(\Pi(\bar
h^*))-\mathbb{E} A^-_l\big).\label{m8}
\end{align}
The first ingredient is the strong concentration of the sums over $n$, which in turn follows from a good control of the fluctuations of the random variables with sum over:
\begin{align}
\mbox{$A_{L}^\pm-\mathbb{E} A_L^\pm$ have exponential tails of $O(1)$}.
\label{ao51ter}
\end{align}
The latter is reminiscent of (\ref{ao26}) and will be addressed in Subsection \ref{SS:Peled}.

\medskip

The second ingredient for the argument for \eqref{ao33} and \eqref{ao51} is the construction of a finite
subset ${\mathcal N}_{\rm fin}$ of the net ${\mathcal N}$ 
with the following two competing properties:
On the one hand, it is large enough so that maximizers satisfy
\begin{align}
\Pi(h^*_{\ge l})\in{\mathcal N}_{\rm fin}&\quad\mbox{with overwhelming probability},\label{ao50bis}\\
\Pi(\bar h^*)\in{\mathcal N}_{\rm fin}&\quad\mbox{with overwhelming probability}\label{ao50ter}
\end{align}
On the other hand, its cardinality is controlled as follows
\begin{align}\label{ao36bis}
\ln\#{\mathcal N}_{\rm fin}\lesssim L/l.
\end{align}
Here comes the auxiliary stochastic result, which we use for $N=L/l$, 
that makes assumptions of the form
(\ref{ao50bis}) \& (\ref{ao50ter}) precise 
and that treats terms of the form of the second r.~h.~s.~terms in (\ref{ao35bis}) and (\ref{ao50}):

\begin{lemma}\label{L:stoch-bis}
Suppose that for every deterministic $\bar h\in{\mathcal N}$
we are given a family $\{A_{n}(\bar h)\}_{n=1,\cdots,N}$
of independent random variables with $A_{n}(\bar h)$ $=_{\rm law}A$. 
Suppose that for every threshold $\nu\in\mathbb{N},\nu\ge2$ there exists ${\mathcal N}_{\nu}$ $\subset{\mathcal N}$ with
\begin{align}\label{ao72}
\ln\#{\mathcal N}_{\nu}\le C_0 N\ln\nu,
\end{align} 
where $C_0$ is the universal constant from Lemma \ref{L:count}.
Then for every $\bar h\in\mathcal{N}_\nu$
\begin{align}
&N^{-1}\sum_{n=1}^{N} A_n(\bar h)\lesssim\|A\|_1(\ln\nu+N^{-1}\bar A)\nonumber\\
&\mbox{with}\quad\|\bar A\|_1\le1\quad\mbox{and}\quad\mbox{$\bar A$ is independent on $\nu$}.\label{mb10}
\end{align}
\end{lemma}

In particular if $\bar h^*\in\mathcal{N},\nu^*\in\mathbb{N}$ are random variables with $\bar h^*\in\mathcal{N}_{\nu^*}$, then
\begin{align*}
\big\|N^{-1}\sum_{n=1}^NA_n(\bar h^*)\big\|_1\lesssim\|A\|_1\|\ln\nu^*\|_1\lesssim\|A\|_1\ln\|\nu^*\|_1,
\end{align*}
where the inequality uses the crude estimate $\|\ln X\|_1\lesssim\ln\|X\|_1$, which can be checked by Lemma~\ref{L:11}.

For the sake of completeness, we display the easy proof of Lemma~\ref{L:stoch-bis}. It is sufficient to prove the weaker
\begin{align}\label{mb1}
\sup_{h\in\mathcal{N}_\nu}N^{-1}\sum_{n=1}^NA_n(\bar h)\lesssim\|A\|_1(\ln\nu+\bar A_\nu)\quad\mbox{with}\quad\|\bar A_\nu\|_1\le1,
\end{align}
where the random variable $\bar A_\nu$ still depends on $\nu$. This dependence can be a posteriori removed by applying Lemma~\ref{L:12}. 

\medskip

In view of Lemma~\ref{L:11}, it is enough to show tails bounds for the l.~h.~s.~. These follow by applying a union bound, Chebychev inequality and independence:
\begin{align*}
\mathbb{P}(\sup_{\bar h\in\mathcal{N}_\nu}N^{-1}\sum_{n=1}^N|A_n(\bar h)|\ge\mu)&\le\sum_{\bar h\in\mathcal{N}_\nu}\mathbb{P}(N^{-1}\sum_{n=1}^N|A_n(\bar h)|\ge\mu)\\&\le\#\mathcal{N}_\nu e^{-N\mu/\|A\|_1+N}\le e^{-N\mu/\|A\|_1+N+C_0N\ln\nu}.
\end{align*}
Applying the change of variables $\mu=\|A\|_1(1+C_0\ln\nu+\tilde\mu/N)$ and using the equivalence of stochastic norms from Lemma~\ref{L:11}, one obtains \eqref{mb1}.

\subsection{Geometry of the net \texorpdfstring{${\mathcal N}$}{N} and counting}\label{SS:bounds}
As discussed in Subsection \ref{SS:tackle}, we need to identify bounds defining the
${\mathcal N}_{\rm fin}\subset{\mathcal N}$ 
that are loose enough such that (\ref{ao50bis}) and (\ref{ao50ter}) hold,
and that are tight enough such that (\ref{ao36bis}) holds.
We first address the latter.
For this, it is helpful to think of the net ${\mathcal N}$ 
defined in (\ref{ao48}) as the integer lattice
in $L/l$-dimensional Euclidean space, which it is in terms of\footnote{modulo the value 
$h(0)/l\in \mathbb{Z}$ at the left endpoint} $\frac{dh}{dx}$, 
as the latter is constant on $((n-1)l,nl)$ for $n=1,\cdots,L/l$ with values in $\mathbb{Z}$.
In particular, the (quadratic part of the) inner product comes from the Dirichlet energy $D/l$.

\medskip

In passing from this integer lattice to its subset ${\mathcal N}$  
we will employ a family of bounds rather than a single bound
since the intersection of the lattice with a single Euclidean ball of radius $L/l$ would
miss the desired (\ref{ao36bis}) by a double logarithm. 
In designing a family of bounds we use the fact that our Euclidean space
naturally is the orthogonal sum of subspaces that correspond to a given dyadic $x$-scale;
this allows us to consider the intersection of the integer lattice
with a corresponding Cartesian product of balls. 
This natural decomposition arises from considering the projection\footnote{For scales 
$l,\rho$, subscripts $h_l,h_\rho$ denote the projection introduced here. For an integer
$n$, the subscript $h_n$ denotes the restricted functions in \eqref{ao57}.} $h_\rho$
of a configuration $h$ to a dyadic scale $\rho$ as given by
\begin{align*}
h_\rho:=h_{\ge\rho}-h_{\ge 2\rho}\quad\mbox{for dyadic}\quad 1\le\rho\le L,
\end{align*}
with the understanding that $h_L=h_{\ge L}$ is the linear interpolation of the boundary data;
this decomposition\footnote{which is reminiscent of a Littlewood-Paley decomposition; on the level of $\frac{dh}{dx}$ it
could be expressed in terms of the Haar basis of $L^2([0,L])$} is indeed orthogonal w.~r.~t.~$D$.

\medskip

Loosely speaking, we implement the idea of considering Cartesian products of balls
by imposing the bounds $D(\bar h_{\rho})\lesssim L$ on all intermediate dyadic
scales $l\le\rho\le L$.
We parameterize these bounds by $\nu\gg 1$ in terms of $D(\bar h_{\rho})/L\le\nu$. 
For the counting, the bounds on the larger-scale components, i.~e.~for $\rho\gg l$, 
matter less since their dimension is smaller;
and in view of the (implicit) maximum over $\rho$, 
it turns out to be necessary for (\ref{ao50bis}) to relax these by 
just imposing $D(\bar h_{\rho})/L\le(\rho/l)^2\nu$,
where the choice of the exponent $2$ is immaterial.

\begin{lemma}[{\cite[Lemma 8]{OPW}}]\label{L:count} 
For $\nu\ge e$ we consider
\begin{align}\label{ao61}
{\mathcal N}_\nu=\{\,\bar h\in{\mathcal N}\,|\,D(\bar h_{\rho})/L\le(\rho/l)^2\nu
\;\;\mbox{for}\;\;l\le\rho\le L,\quad|\bar h(0)|/L\le\nu\,\}.
\end{align}
Then there exists a universal constant $C_0$ such that
\begin{align*}
\ln \# {\mathcal N}_\nu\le C_0(L/l)\ln\nu.
\end{align*}
\end{lemma}

We now turn to (\ref{ao50bis}) and (\ref{ao50ter}). In view of the choice (\ref{ao61}),
(\ref{ao50bis}) will be satisfied provided the maximizer $h^*_{y_0,y_1}$ of $(W-D)/L$
with boundary conditions $y_0,y_1$ satisfies (as always with overwhelming probability)
\begin{align}\label{ao64}
\sup_{|y_0|,|y_1|\le L/2}\,D(\Pi(h^*_{y_0,y_1,\rho}))/L=O(1)\quad\mbox{for dyadic}\;l\le \rho\le L.
\end{align}
We first note that the projection $\Pi$ in (\ref{ao64}) is irrelevant:
Since $\bar h_\rho-\Pi(\bar h)_\rho$ is linear in $[(n-1)\rho,n\rho]$ with
boundary conditions in $[-l/2,l/2]$, one has
\begin{align*}
D_{[(n-1)\rho,n\rho]}(\bar h_\rho-\Pi(\bar h)_\rho)\le l^2/(2\rho).
\end{align*}
Together with the triangle inequality for $h\mapsto\sqrt{D(h)/L}$, we have for any (coarse-grained) configuration $\bar h$
%
$\sqrt{D(\Pi(\bar h))/L}$ $\le\sqrt{D(\bar h)/L}$ $+l/(\sqrt{2}\rho)$,
%
hence the control (\ref{ao64}) follows from
\begin{align}\label{ao65}
\sup_{|y_0|,|y_1|\le L/2}\,D(h^*_{y_0,y_1,\rho})/L=O(1)\quad\mbox{for dyadic}\;l\le \rho\le L.
\end{align}

More precisely, 

\begin{remark}\label{rmk1-bis}
if $\mathcal{N}_\nu$ is defined in \eqref{ao61}, then for every $|y_0|,|y_1|\le L/2$,
\begin{align}
\bar h^*_{y_0,y_1}\in\mathcal{N}_{\nu^*}\quad&\mbox{with}\quad \nu^*\lesssim1+\max_{l\le\rho<L}(l/\rho)^2\max_{|y_0|,|y_1|\le L/2}D(h^*_{y_0,y_1,\rho})/L,\nonumber\\
&\mbox{so that}\quad\|\nu^*\|_1\lesssim1+\max_{l\le\rho<L}\|\max_{|y_0|,|y_1|\le L/2}D(h^*_{y_0,y_1,\rho})/L\|_1.\label{m6-bis}
\end{align}
where in the last step with estimated the maximum over $\rho$ by the sum and used the geometric decay of the weight $(l/\rho)^2$.
\end{remark}

\medskip

We are finally able to make the heuristic outlined in Subsection~\ref{SS:count}
rigorous. Starting from \eqref{m7} and \eqref{m8}, we combine Lemmas~\ref{L:stoch-bis} and
\ref{L:count} together with \eqref{m6-bis} to get the following precise version of
the sub- and super-additivity inequality.

\begin{lemma}\label{L:subsuper}
\begin{align*}
\mathbb{E}A^+_L-\mathbb{E}A^+_{L/l}-\mathbb{E}A^+_l\lesssim&(\|A^+_l-\mathbb{E}A^+_l\|_1+1)
\\&\cdot\ln(e+\max_\rho\|\max_{|y_0|,|y_1|\le L/2}D(h^*_{y_0,y_1,\rho})/L\|_1),\\
\mathbb{E}A^-_L-\mathbb{E}A^-_{L/l}-\mathbb{E}A^-_l\gtrsim&-(\|A^-_l-\mathbb{E}A^-_l\|_1+1)
\\&\cdot\ln(e+\max_\rho\|\max_{|y_0|,|y_1|\le L/2}D(h^*_{y_0,y_1,\rho})/L\|_1).
\end{align*}
\end{lemma}


\subsection{Equipartition of the Dirichlet energy over logarithmic scales.}

We now argue how to control the scale-by-scale Dirichlet energy of the
maximizers
\begin{align*}
\max_\rho\|\sup_{|y_0|,|y_1|\le L/2}D(h_{y_0,y_1,\rho}^*)/L\|_1
\end{align*}
in terms of the Orlicz norms of the mid-point deviations
\begin{align}\label{ao67}
H_L:=\sup_{|y_0|,|y_1|\le
L/2}\,|h^*_{y_0,y_1}(L/2)-{\textstyle\frac{1}{2}}(y_0+y_1)|/L.
\end{align}
%
This quantity will be estimated in Subsection \ref{SS:Peled}, cf.~Corollary \ref{C:1}.
We start from the representation of the l.~h.~s.~of (\ref{ao65}), 
which for convenience we express in terms of $l=2\rho$:
\begin{align}
D(h_{l/2})/L=(L/l)^{-1}\sum_{n=1}^{L/l}(h_n(l/2)/l)^2\nonumber\\
\mbox{where}\quad h_n=(h-h_{\ge l})(\cdot+(n-1)l).\label{mb9}
\end{align}

\medskip

We'd like to apply this representation to the maximizer $h=h^*_{y_0,y_1}$ of $(W-D)/L$ with
inhomogeneous Dirichlet boundary conditions $y_0,y_1$. Note that, because of the local form of the action,
\begin{align}
&\mbox{the restriction $h_{y_0,y_1}^*|_I$ to some interval $I$ is a maximizer}\nonumber\\&\mbox{with respect to its own boundary conditions.}\label{m45}
\end{align}
Thus, conditioned on $\Pi(h^*_{y_0,y_1,\ge l})$ $=\bar h$ for some deterministic $\bar h\in{\mathcal N}$,
the configuration
$\tilde h_n^*$ $=(h^*_{y_0,y_1}-\Pi(h^*_{y_0,y_1,\ge l}))(\cdot+(n-1)l)$ maximizes $W_n-D$,
under some inhomogeneous Dirichlet boundary conditions\footnote{namely
$\Pi(h^*_{y_0,y_1,\ge l})-h^*_{y_0,y_1,\ge l}$} constrained to the $y$-interval $[-l/2,l/2]$,
where the functional $W_n(\tilde h)$ $=W_{[(n-1)l,nl]}(\bar h+\tilde h(\cdot-(n-1)l))$
$-W_{[(n-1)l,nl]}(\bar h)$ has the same law as $W_{[0,l]}$.
Hence 
%
%
we obtain

\begin{lemma}\label{L:dir-bis} For every deterministic $\bar h\in{\mathcal N}$ there exists a family
$\{H_{l,n}(\bar h)\}_n$ of independent random variables that are identically distributed
according to $H_l$ such that 
if $h^*_{y_0,y_1}$ denotes the maximizer of $(W-D)/L$ with inhomogeneous
Dirichlet boundary conditions $y_0,y_1$ we have
\begin{align*}
\max_{|y_0|,|y_1|\le L/2}D(h^*_{y_0,y_1,l/2})/L
\le(L/l)^{-1}\sum_{n=1}^{L/l}H_{l,n}^2(\Pi(h^*_{y_0,y_1,\ge l})).
\end{align*}
%
%
%
\end{lemma}

Hence by the informal argument from Subsection \ref{SS:tackle},
we may inductively in $\log_2 L/l$ derive (\ref{ao64}) via its equivalent form (\ref{ao65})
from the fact that (\ref{ao67}) with $L$ replaced by $l$ is O(1). More precisely, we
can apply to same reasoning used in the proof of Lemma~\ref{L:subsuper} to get
\begin{align}\label{m25}
\|\max_{|y_0|,|y_1|\le L/2}&D(h^*_{y_0,y_1,l/2})/L\|_1\nonumber\\&\lesssim\|H_l^2\|_1
\ln(e+\max_{l\le\rho<L}\|\max_{|y_0|,|y_1|\le L/2}D(h^*_{y_0,y_1,\rho})/L\|_1).
\end{align}
Taking the maximum in $l$, it turns into
\begin{align*}
\max_l\|\max_{|y_0|,|y_1|\le L/2}&D(h^*_{y_0,y_1,l})/L\|_1\nonumber\\&\lesssim\big(\max_l\|H_l^2\|_1\big)
\ln(e+\max_l\|\max_{|y_0|,|y_1|\le L/2}D(h^*_{y_0,y_1,l})/L\|_1)
\end{align*}
which allows to buckle\footnote{Note that the qualitative finitess of
$\|\max_{|y_0|,|y_1|\le L/2}D(h^*_{y_0,y_1,l})\|_1$ can be deduced from that of
$\{\|H_l^2\|_1\}_l$ by inductively applying \eqref{m25}.\label{fn1}}

\begin{lemma}\label{L:dir}
\begin{align*}
\max_l\|\max_{|y_0|,|y_1|\le L/2}D(h^*_{y_0,y_1,l})/L\|_1\lesssim
\max_{l}\|H_l^2\|_1\ln(e+\max_l\|H_l^2\|_1).
\end{align*}
\end{lemma}


\subsection{Dembin-Elboim-Hadas-Peled on the control of fluctuations}\label{SS:Peled}

We now turn to justification of the hypothesis of $O(1)$ exponential tails for the fluctuations
of  the maximal energies $A^{\pm}_L$ used in Subsection \ref{SS:tackle}, 
see (\ref{ao51ter}),
and the $O(1)$ exponential tails for the height $H_L$ at the midpoint used in Subsection
\ref{SS:bounds}, see (\ref{ao67}). 

\medskip

With respect to the Cameron-Martin space of our white noise $\xi$, 
which is $L^2(\mathbb{R}_{x}\times\mathbb{R}_y)$, and for a given configuration $h$,
the functional $\xi\mapsto (W-D)(h)/L$ is Lipschitz continuous with constant  
$\sqrt{L^{-2}\int_0^L dx|h(x)|}$.
Let $h^*_{y_0,y_1}$ denote the maximizer of $(W-D)/L$ constrained to Dirichlet boundary data
$y_0,y_1$; if we had a deterministic and uniform bound on its $L^1$ norm, in form of
\begin{align}\label{ao60}
\mbox{naively}\quad L^{-2}\int_0^L dx|h^*_{y_0,y_1}(x)|\le M
\quad\mbox{for all}\;|y_0|,|y_1|\le\frac{L}{2},
\end{align}  
we could constrain the maximum in $A_L^{\pm}$ to the corresponding set of configurations. Note also that (cf.~\eqref{m41})
\begin{align*}
L^{-2}\int_0^L dx|a_{y_0,y_1}(x)|\le\frac{1}{2}
\quad\mbox{for all}\;|y_0|,|y_1|\le\frac{L}{2}.
\end{align*}
Since taking maxima and minima over families of functions of $\xi$ preserves
the Lipschitz norm, $\xi\mapsto A_L^{\pm}$ would have Lipschitz constant $\le 
\sqrt{M+1/2}$. Hence by Gaussian concentration we would expect
\begin{align*}
\mbox{naively}\quad|A_L^{\pm}-\mathbb{E}A_L^{\pm}|/\sqrt{M+1/2}\quad\mbox{has Gaussian tails of $O(1)$}.
\end{align*}
This remains true if the bound (\ref{ao60}) just holds with overwhelming probability;
the tails relate as expected, and we state the intermediate result in the required Orlicz norm:

\begin{lemma}\label{L:1}
\begin{align*}
\|A^\pm_L-\mathbb{E}A^\pm_L\|_{3/2}^2
&\lesssim\big\|{\textstyle\sup_{|y_0|,|y_1|\le L/2}}{\textstyle L^{-2}\int_0^L dx}
|h^*_{y_0,y_1}(x)|\big\|_3+1.
\end{align*}
\end{lemma}

We note that maximizers $h^*_{y_0,y_1}$ 
of just $-D$ for given boundary conditions $y_0,y_1$ would be linear;
the argument, which is due to \cite[Theorem 1.4]{DEHP} and is also used in \cite[Lemma 3.6]{DEP}, is based on the first variation in direction of 
some infinitesimal (continuous and piecewise linear)
configuration $\dot h$ with vanishing boundary conditions
and involves the corresponding
bilinear form 
\begin{align*}
D(h^*_{y_0,y_1},\dot h):=\frac{1}{2}\int_0^Ldx\frac{dh^*_{y_0,y_1}}{dx}\frac{d\dot h}{dx}.
\end{align*}
In the presence of the non-differentiable functional $h\mapsto W(h)$, it is 
better to pass to finite variations like $h^*_{y_0,y_1}-\dot h$;
using the maximality in form of $(W$ $-D)(h^*_{y_0,y_1}-\dot h)$ $\le (W-D)(h^*_{y_0,y_1})$
we obtain by expanding the square
\begin{align*}
2D(h^*_{y_0,y_1};\dot h)\le W(h^*_{y_0,y_1})-W(h^*_{y_0,y_1}-\dot h)+D(-\dot h).
\end{align*}

\medskip

In order to leverage the fact that the functional $h\mapsto W_{\dot h}(h)$ 
$:=W(h-\dot h)-W(-\dot h)$
has the same law as $W$ (see~\eqref{m33}), we rewrite the latter as
\begin{align*}
2D(h^*_{y_0,y_1};\dot h)
\le(W-D)(h^*_{y_0,y_1})-(W_{\dot h}-D)(h^*_{y_0,y_1})-(W-D)(-\dot h).
\end{align*}
By definitions \eqref{m38} and \eqref{m39} this implies
\begin{align*}
&2\max_{|y_0|,|y_1|\le L/2}D(h^*_{y_0,y_1};\dot h)/L
\le A_L^+-A_{L,0}^-+X_0-(W-D)(-\dot h)/L\nonumber\\
&\mbox{where}\quad A_{L,0}^-=_{\rm law}A_L^-,\quad\mbox{and for some $X_0$ with}\;\|X_0\|_2\lesssim1,
\end{align*}
where $X_0$ accounts for the action of the linear parts and is controlled by~\eqref{m40}. We upgrade it to
\begin{align}
&\max_{|y_0|,|y_1|\le L/2}|D(h^*_{y_0,y_1};\dot h)|/L
\le A_L^+-{\textstyle\frac{1}{2}}A_{L,0}^--{\textstyle\frac{1}{2}}A_{L,1}^-+{\textstyle\frac{1}{2}}X_0+{\textstyle\frac{1}{2}}X_1\nonumber\\
&\qquad\qquad\qquad\qquad\qquad\quad-{\textstyle\frac{1}{2}}(W-D)(\dot h)/L-{\textstyle\frac{1}{2}}(W-D)(-\dot h)/L
\nonumber\\&\mbox{where}\;A_{L,1}^-=_{\rm law}A_L^-,\quad\mbox{and for some $X_1$ with}\;\|X_1\|_2\lesssim1.\label{m1}
\end{align}
It remains to take $\dot h$ $=t G(x,\cdot)$, where $x$ is a lattice point and
$G$ is the Green function for the Dirichlet Laplacian, which happens to be continuous
and piecewise linear and thus admissible, and is made such that 
$D(h^*_{y_0,y_1};\dot h)$ $=th^*_{y_0,y_1}(x)$. The result below is
derived by optimizing in $t$. This is not immediate because of the implicit
dependence of $A_{L,0}^-,A_{L,1}^-,X_0,X_1$ on $t$. A detailed proof is given in 
Section~\ref{S:det}.

\begin{lemma}\label{L:2} For any lattice point $x$
\begin{align*}
\big\|\max_{|y_0|,|y_1|\le L/2}|h^*_{y_0,y_1}(x)|/L\big\|_3^2
\lesssim\mathbb{E}A^+_L-\mathbb{E}A^-_L+\|A^\pm_L-\mathbb{E}A^\pm_L\|_{3/2}+1.
\end{align*}
\end{lemma}

Obviously, Lemmas \ref{L:1} and \ref{L:2} combine to

\begin{corollary}\label{C:1} 
\begin{align*}
\big\|H_L\big\|_3^2
+\big\|A^\pm_L-\mathbb{E}A^\pm_L\big\|_{3/2}^4
\lesssim \mathbb{E}A^+_L-\mathbb{E}A^-_L+1.
\end{align*}
\end{corollary}

We learn from Corollary \ref{C:1} that the tasks of 
showing that $\mathbb{E}A_L^+$ is asymptotically not larger than $\mathbb{E}A_L^-$
is tied to the task of establishing the fluctuation tails (\ref{ao51ter}),
and even the mid-point estimate (\ref{ao67}). This coupling arises from the need
to take the supremum over Dirichlet boundary data.  


\subsection{Buckling and proof of the theorems}

Let us start by pointing out that the quantities in the previous subsections are
all (qualitatively) finite. Indeed, in the appendix we prove that
\begin{align}\label{m20}
\|A^+_L\|_{3/2}<\infty\quad\mbox{which implies}\quad\mathbb{E}A_L^\pm,
\|A_L^\pm-\mathbb{E}A_L^\pm\|_{3/2}<\infty.
\end{align}
Consecutively applying Lemmas~\ref{L:2} and~\ref{L:dir} (more precisely
footnote~\ref{fn1}) gives
\begin{align*}
\|H_L\|_3<\infty\quad\mbox{and}\quad\|\sup_{|y_0|,|y_1|\le L/2}D(h^*_{y_0,y_1,l})\|_1<\infty.
\end{align*}

\medskip

In the sequel, we combine the results from Lemmas~\ref{L:pm},~\ref{L:subsuper},~\ref{L:dir}
and Corollary~\ref{C:1} to obtain that the following quantities are $\lesssim1$:
\begin{align}\label{m13}
\|H_L\|_3,\|A_L^\pm-\mathbb{E}A_L^\pm\|_{3/2},\|\sup_{|y_0|,|y_1|\le
L/2}D(h^*_{y_0,y_1,l})/L\|_1,\mathbb{E}A_L^+-\mathbb{E}A_L^-.
\end{align}
Starting from Lemmas~\ref{L:pm} and \ref{L:subsuper}, we first show that
\begin{align}
\mathbb{E}A_L^+-\mathbb{E}A_L^-\lesssim&(\max_l\|A_l^\pm-\mathbb{E}A_l^\pm\|_1+1)\ln(e+\max_l\|H_l\|_1).
\label{m11}
\end{align}
To this purpose, let us denote the r.~h.~s.~by $\Delta$ and let $C<\infty$ denote a universal constant
(the value of which may change from line to line). From Lemma~\ref{L:pm} and~\ref{L:dir},
\begin{align*}
\mathbb{E}A_L^+-\mathbb{E}A_L^-\le\mathbb{E}A_L^+
-{\textstyle\frac{1}{2}}\mathbb{E}A_{L/2}^+-{\textstyle\frac{1}{2}}\mathbb{E}A_{L/4}^{-}+C\Delta.
\end{align*}
Using Lemma~\ref{L:subsuper} with $l=2$ and $l=4$ allows for the replacements
\begin{align*}
{\textstyle\frac{1}{2}}\mathbb{E}A_L^+\le{\textstyle\frac{1}{2}}
\mathbb{E}A_{L/2}^++{\textstyle\frac{1}{2}}\mathbb{E}A_2^++C\Delta,\qquad
{\textstyle\frac{1}{2}}\mathbb{E}A_L^+\le{\textstyle\frac{1}{2}}
\mathbb{E}A_{L/4}^++{\textstyle\frac{1}{2}}\mathbb{E}A_4^++C\Delta.
\end{align*}
Since $\mathbb{E}A_2^+,\mathbb{E}A_4^+\lesssim1$, cf.~\eqref{m20}, from the
previous two lines one deduces the (almost) contraction
\begin{align*}
\mathbb{E}A_L^+-\mathbb{E}A_L^-\le{\textstyle\frac{1}{2}}(\mathbb{E}A_{L/4}^+-
\mathbb{E}A_{L/4}^-)+C\Delta.
\end{align*}
Iterating it, we get \eqref{m11}.

\medskip

Addressing now \eqref{m13}, we apply Corollary~\ref{C:1} and \eqref{m11} to the effect of
\begin{align*}
&\|H_L\|_3^2+\|A_L^\pm-\mathbb{E}A_L^\pm\|_{3/2}^4\lesssim(\max_l\|A_l^\pm-
\mathbb{E}A_l^\pm\|_1+1)\ln(e+\max_l\|H_l\|_1),
\end{align*}
so that taking the maximum over the scales on the l.~h.~s.,
\begin{align*}
\max_l\|H_l\|_3^2+\max_l&\|A_l^\pm-\mathbb{E}A_l^\pm\|_{3/2}^4\\
&\lesssim(\max_l\|A_l^\pm-\mathbb{E}A_l^\pm\|_{3/2}+1)\ln(e+\max_l\|H_l\|_3^2).
\end{align*}
Since the function $x^2+y^4-C(y+1)\ln(e+x)$ is coercive, we obtain the bound for
the first two quantities in \eqref{m13}. The other two are controlled by using
Lemma~\ref{L:dir} and \eqref{m11}.

\medskip

Here comes the proof of Theorem~\ref{T:1}. Since $A_L$ is controlled from above
and from below by $A_L^\pm$, see~\eqref{m26}, we deduce from Lemma~\ref{L:subsuper} and~\eqref{m13}
\begin{align}\label{m14}
|\mathbb{E}A_L-\mathbb{E}A_l-\mathbb{E}A_{L/l}|\lesssim1.
\end{align}
This implies the quantitative convergence
\begin{align}\label{m12}
|\mathbb{E}A_L-a^*\ln L|\lesssim1\quad
\mbox{for some constant $a^*\in[0,\infty)$},
\end{align}
see Section~\ref{S:det} for the details. From the sandwiching~\eqref{m26} and the estimates
\begin{align*}
\mathbb{E}A_L^+-\mathbb{E}A_L^-,\|A_L^\pm-\mathbb{E}A_L^\pm\|_{3/2}\lesssim1\quad\mbox{(see~\eqref{m13})},
\end{align*}
we deduce that also $\|A_L-\mathbb{E}A_L\|_{3/2}\lesssim1$, which allows to upgrade
the asymptotics of the expectations \eqref{m12} to the result of
Theorem~\ref{T:1}.

\medskip

We conclude with the proof of Theorem~\ref{T:2}. By triangle inequality, it is
enough to treat the case $l'=1$. Because of the decomposition \eqref{ao32}
applied to $h=h^*$, we need to show
\begin{align}\label{m46}
\|(L/l)^{-1}\sum_{n=1}^{L/l}(W_n(h_n^*)-D(h_n^*))/l-a^*\ln l\|_{3/2}\lesssim1.
\end{align}
Because of~\eqref{m45}, we deduce that the functions $h_n^*$ obtained from $h^*$ maximize $W_n-D$ with homogeneous Dirichlet boundary data.
Hence the observations \eqref{m16} and \eqref{m17} allow to control
$W_n(h_n^*)-D(h_n^*)$ from above and below by
\begin{align*}
A_{l,n}^-(\Pi(h^*_{\ge l}))\le(W_n(h_n^*)-D(h_n^*))/l\le A_{l,n}^+(\Pi(h_{\ge l}^*)),
\end{align*}
so that the statement \eqref{m46} reduces to
\begin{align*}
\|(L/l)^{-1}\sum_{n=1}^{L/l} A^\pm_{l,n}(\Pi(h^*_{\ge l}))-a^*\ln l\|_{3/2}\lesssim1.
\end{align*}
By Theorem~\ref{T:1} and \eqref{m13}, we can replace $a^*\ln l$ by
$\mathbb{E} A_l^\pm$. By Jensen's inequality in $n$ with exponent
$3/2$, it is enough to show
\begin{align*}
\big\|(L/l)^{-1}\sum_{n=1}^{L/l}| A^\pm_{l,n}(\Pi(h^*_{\ge
l}))-\mathbb{E} A^\pm_l|^{3/2}\big\|_1\lesssim1,
\end{align*}
which follows by applying Lemma~\ref{L:stoch-bis} (as in the proof of
Lemma~\ref{L:subsuper}), together with the now established estimates
\begin{align*}
&\| A^\pm_l(\bar h)-\mathbb{E} A^\pm_l\|_{3/2}&&\mbox{(see~\eqref{m13} with $l$ replacing $L$),}\\
&\|D(h^*_{l})/L\|_1\lesssim1&&\mbox{(implied by \eqref{m13})}.
\end{align*}
%


\section{Additional details}\label{S:det}

{\sc Proof of Lemma~\ref{L:2}.} Given a lattice point $x$ and the Green function
$G(x,\cdot)$ defined by
\begin{align*}
G(x,0)=G(x,L)=0\qquad\mbox{and}\qquad\frac{d^2 G(x,x')}{dx'^2}=\delta_{x}(dx'),
\end{align*}
let us choose $\dot h(\cdot)=tG(x,\cdot)$ in \eqref{m1}. It is not difficult to
see that
\begin{align*}
-(W-D)(\dot h)/L-(W-D)(-\dot h)/L\le Ct^2+CB_t,
\end{align*}
where $\{B_t\}_{t\ge0}$ is a standard Brownian motion and $C$ is universal. Hence,
\begin{align}\label{m3}
t\sup_{|y_0|,|y_1|\le L/2}|h_{y_0,y_1}^*(x)|/L\le
A_L^+-{\textstyle\frac{1}{2}}A_{L,0}^--{\textstyle\frac{1}{2}}A_{L,1}^-+Ct^2+CB_t+X,
\end{align}
with $X:={\textstyle\frac{1}{2}}X_0+{\textstyle\frac{1}{2}}X_1$ which thus satisfies $\|X\|_2\lesssim1$.

\medskip

To simplify the r.~h.~s.~, we observe that the problem $A_L$ with $L=2$ satisfies
\begin{align}
&A_2=_{\text{law}}\max_{t}B_t-t^2\quad\mbox{so that}\nonumber\\&
B_t\le t^2+\bar B\quad\mbox{with}\quad\|\bar B\|_{3/2}\lesssim1.\label{mb3}
\end{align}
The additional quadratic term $t^2$ can by absorbed into $Ct^2$, up to modifying the constant $C$.

\medskip

Notice that the random variable $A_L^+-{\textstyle\frac{1}{2}}A_{L,0}^--{\textstyle\frac{1}{2}}A_{L,1}^-$ depends on $t$ and, separating expectations and fluctuations, it satisfies for every fixed $t$ 
\begin{align}\label{mb2}
\|A_L^+-{\textstyle\frac{1}{2}}A_{L,0}^--{\textstyle\frac{1}{2}}A_{L,1}^-\|_{3/2}
\lesssim(\mathbb{E}A_L^+-\mathbb{E}A_L^-)+\|A_L^\pm-\mathbb{E}A_L^\pm\|_{3/2}.
\end{align}
Starting from \eqref{m3}, we would like to take $t=c\sup_{|y_0|,|y_1|\le L/2}|h^*_{y_0,y_1}(x)|/L$ with $c\ll C^{-1}$ in order to absorb $Ct^2$ into the l.~h.~s.~. Nevertheless, since $t$ is random, the estimate \eqref{mb2} does not hold anymore. In order to overcome the problem, we discretize $t$ and pick
\begin{align*}
&t^*:=c\big\lfloor R^{-1/2}H\big\rfloor R^{1/2}\quad\mbox{with}\quad c\ll C^{-1}
\end{align*}
and $R,H$ are the short for
\begin{align}\label{mb6}
&R:=(\mathbb{E}A_L^+-\mathbb{E}A_L^-)+\|A_L^\pm-\mathbb{E}A_L^\pm\|_{3/2},\nonumber\\& H:=\sup_{|y_0|,|y_1|\le L/2}|h^*_{y_0,y_1}(x)|/L.
\end{align}
In view of the discreteness of $t^*$, we can apply \eqref{mb2} and Lemma~\ref{L:12} to obtain
\begin{align}\label{mb4}
&A_L^+-{\textstyle\frac{1}{2}}A_{L,0}^--{\textstyle\frac{1}{2}}A_{L,1}^-\lesssim R(\ln^{2/3}(R^{-1/2}H)+\bar A)\quad\mbox{with}\quad\|\bar A\|_{3/2}\le1.
\end{align}

\medskip

Gathering \eqref{m3}, \eqref{mb3} and \eqref{mb4}, we conclude that
\begin{align*}
&\lfloor R^{-1/2}H\rfloor R^{1/2}H\lesssim R\ln^{2/3}(R^{-1/2}H)+R\bar A+\bar B+X.
\end{align*}
Dividing by $R$, it is not difficult to see that the $\log$ in the r.~h.~s.~can be absorbed into the l.~h.~s.~, hence we get
\begin{align}\label{mb5}
H^2\lesssim R(1+\bar A)+\bar B+X.
\end{align}
Taking the norm $\|\cdot\|_{3/2}$ on both sides, we conclude.
\qed

\medskip


{\sc Proof of \eqref{m12} from \eqref{m14}.}
Let $L=l^k l'$ with $k\ge 1$ and $1\le l'<l$. Iterating
\eqref{m14} for $k$ times,
\begin{align*}
|\mathbb{E}A_L-k\mathbb{E}A_l-\mathbb{E}A_{l'}|\lesssim k.
\end{align*}
Dividing by $\ln L=k\ln l+\ln l'$ and letting $L\uparrow\infty$ 
\begin{align}\label{m15}
|\limsup_L\frac{\mathbb{E}A_L}{\ln L}-\frac{\mathbb{E}A_l}{\ln
l}|\lesssim\frac{1}{\ln l}.
\end{align}
Taking now the liminf as $l\uparrow\infty$
\begin{align*}
\limsup_L\frac{\mathbb{E}A_L}{\ln L}=\liminf_{l}\frac{\mathbb{E}A_l}{\ln l}\quad
\mbox{so that there exists}\quad a^*:=\lim_{L}\frac{\mathbb{E}A_L}{\ln L}.
\end{align*}
Substituting $a^*$ in \eqref{m15}, we have obtained \eqref{m12}.
\qed


\section{Acknowledgements}

We thank Xiaopeng Cheng for helpful discussions. The authors acknowledge funding by the
Deutsche Forschungsgemeinschaft (DFG, German Research Foundation) – CRC/TRR 388
"Rough Analysis, Stochastic Dynamics and Related Fields" – Project ID 516748464.

\begin{appendix}

\section{Auxiliary lemmas}


{\sc Proof of \eqref{m40}.} Since the linear functions are uniformly 1-Lipschitz, we can discard the Dirichlet energy and just control the noise $W$. Let us use the shorter notation $z=(y_0,y_1)$ and
$w(z):=W(a_{y_0,y_1})$. By scaling $L=1$, so that
$z\in[-{\textstyle\frac{1}{2}},{\textstyle\frac{1}{2}}]^2$. One starts from
\begin{align*}
\mathbb{E}^{1/p}|w(z)-w(z')|^p\lesssim_p|z-z'|^{1/2}\quad\mbox{for all
$p\in[1,\infty)$}
\end{align*}
which can be established for $p=2$ using the definition of the white noise $\xi$
(which enters the definition of $W$) and is extended to $p\neq 2$ by Gaussianity.
For $d=2$, the dimension of the $z$ space, let us introduce
$\alpha:=1/2-d/p$. We divide the inequality above by the r.~h.~s.~and raise to
power $p$. Integrating in $z,z'$ and applying Fubini's theorem yields the
integrability of the Besov semi-norm $W^{\alpha,p}$ (to the $p$-th power)
\begin{align*}
\mathbb{E}\iint\big(\frac{|w(z)-w(z')|}{|z-z'|^\alpha}\big)^p\frac{1}{|z-z'|^d}dzdz'<\infty.
\end{align*}
By Besov embedding (for $\alpha p>d$) $w$ belongs to the homogeneous
H\"older space $C^\beta$ with $\beta=\alpha-d/p$. Since $w(0)=0$ and $z$ lives in a
compact space, this norm controls the $L^\infty$ norm, so that
\begin{align*}
\mathbb{E}\sup_{z}|w(z)|^p<\infty.
\end{align*}
Having a control on the $p$-th moment (any $p\ge1$ would suffice), the finiteness of the Orlicz norm follows
from Borell's inequality (being $w$ Gaussian)
\begin{align*}
\big\|\sup_{z}|w(z)|-\mathbb{E}\sup_{z}|w(z)|\big\|_2\lesssim\sup_z\mathbb{E}^{1/2}|w(z)|^2<\infty.
\end{align*}
\qed


{\sc Proof of \eqref{m20}.}
Let us start from the constrained problem
\begin{align*}
X(\hat D):=\sup\big\{W(h)/L~\big|~D(h)\le\hat
DL~\mbox{and}~|h(0)|,|h(L)|\le L/2\big\}.
\end{align*}
The above proof of \eqref{m40} extends to Gaussian processes on compact subsets
of $\mathbb{R}^d$ for any $d$ and gives the qualitative finiteness
\begin{align}\label{m21}
\|X(1)\|_2<\infty.
\end{align}
Moreover the scaling (cf.~\eqref{ao68})
\begin{align*}
h=\hat D^{1/2}\hat h\quad\mbox{implies for $\hat D\ge1$}\quad X(\hat
D)\le_{\text{law}}\hat D^{1/4}X(1).
\end{align*}

\medskip

Given $\nu\ge1$, we distinguish the cases $D(h)/L\in[(n-1)\nu,n\nu]$
to get
\begin{align*}
\mathbb{P}(A^+_L\ge\nu)\le\sum_{n=1}^\infty\mathbb{P}(X(n\nu)\ge
n\nu)\le\sum_{n=1}^\infty\mathbb{P}(X(1)\ge n^{3/4}\nu^{3/4}).
\end{align*}
Using Lemma~\ref{L:11} and \eqref{m21}, the sum is finite and bounded by
$\exp(-c\nu^{3/2})$, hence $\|A_L^+\|_{3/2}<\infty$.
\qed


\begin{lemma}[{\cite[Lemma 5]{OPW}}]\label{L:11}
For $s\in[1,\infty)$ and a random variable $X$
\begin{align*}
&\ln\mathbb{P}(X\ge\nu)\lesssim-(\nu/\|X\|_s)^s\quad\mbox{for
all}~\nu\gg\|X\|_s;\\
&\mbox{if}\quad\ln\mathbb{P}(X\ge\nu)\le-\nu^s\quad\mbox{for}~\nu\ge\nu_0>0\quad\mbox{then}
\quad\|X\|_s-\nu_0\lesssim1.
\end{align*}
\end{lemma}

\begin{lemma}\label{L:12}
For every collection of random variables $\{X_n\}_{n\ge1}$, there exists $\bar X$ such that for every $n$
\begin{align*}
X_n/\|X_n\|_s\lesssim\ln^{1/s}n+\bar X\quad\mbox{and}\quad\|\bar X\|_s\le1.
\end{align*}
\end{lemma}

{\sc Proof.} By scaling, we assume $\|X_n\|_s=1$. The proof follows by applying a union bound and Chebyshev inequality. Indeed, defining $\bar X:=\sup_n\bar X_n-C\ln^{1/s}n$ for some $C>0$ and picking a threshold $\mu\gg1$,
\begin{align*}
\mathbb{P}(\bar X\ge\mu)&\le\sum_{n\in\mathbb{N}}\mathbb{P}(X_n-C\ln^{1/s}n\ge\mu)\\
&\le\sum_{n\in\mathbb{N}}\mathbb{P}(|X_n|^s\ge C^s\ln n+\mu^s)\le\sum_{n\in\mathbb{N}}e^{-\mu^s+C^s\ln n+1}\lesssim e^{-\mu^s},
\end{align*}
where the last inequality holds provided $C>1$. We conclude by applying the equivalence of stochastic norms in Lemma~\ref{L:11}.
\qed

\section{Regularity of almost maximizers}

The goal of this section is to extend the scale-by-scale regularity to almost maximizers of the action $W-D$, which will be a useful result in \cite{COP}. To this end, we measure by how much a configuration $h:[0,L]\to\mathbb{R}$ is not a maximizer (with respect to its own boundary condition) by introducing the energy gap
\begin{align*}
\Lambda(h):=\max_{h'(0)=h(0),~h'(L)=h(L)}(W(h')-D(h'))/L-(W(h)-D(h))/L.
\end{align*}

\begin{lemma}
There exists a random variable $\bar D$ with $\|\bar D\|_{3/2}\le1$ such that for every configuration $h$ with $|h(0)|,|h(L)|\le0$,
\begin{align*}
\max_l D(h_l)/L\lesssim\bar D+\Lambda(h).
\end{align*}
\end{lemma}

{\sc Proof.} As in Lemma~\ref{L:dir}, we prove this inequality inductively from large to small scales. In view of the constraint on the boundary conditions $|h(0)|,|h(L)|\le L/2$, the inequality for the larger scale,
\begin{align}\label{mb8}
D(h_{L/2})/L\lesssim\bar D+\Lambda(h),
\end{align}
amounts to the bound on the middle point,
\begin{align}\label{mb7}
(h(L/2)/L)^2\lesssim\bar D+\Lambda(h)\quad\mbox{provided}~|h(0)|,|h(L)|\le L/2.
\end{align}
This follows by adapting the proof of Lemma~\ref{L:2}. Indeed, with the same steps, one obtains the analogous of \eqref{m3}
\begin{align*}
t|h(L/2)|/L\le A_L^+-{\textstyle\frac{1}{2}}A_{L,0}^--{\textstyle\frac{1}{2}}A_{L,1}^-+Ct^2+CB_t+X+\Lambda(h),
\end{align*}
which leads to the variant of \eqref{mb5}
\begin{align*}
&(|h(L/2)|/L)^2\lesssim R(1+\bar A)+\bar B+X+\Lambda(h)\\&\mbox{where}\quad\|\bar A\|_{3/2},\|\bar B\|_{3/2},\|X\|_2\le1\quad\mbox{and}\quad\mbox{$R$ is defined in \eqref{mb6}}.
\end{align*}
From \eqref{m13}, we have learned that $R\lesssim1$, so that the last inequality implies \eqref{mb7} and hence \eqref{mb8}.

\medskip

Let us now iterate \eqref{mb6} to intermediate scales $l$. We recall that \eqref{mb9} allows to write $D(h_{l/2})$ in terms of local $L^\infty$ bounds. Hence by applying local versions of \eqref{mb8} and the same reasoning that led to Lemma~\ref{L:dir-bis}, we obtain
\begin{align}\label{mb13}
D(h_{l/2})/L\lesssim (L/l)^{-1}\sum_{n=1}^{L/l}\bar D_n(\bar h)+\Lambda(h|_{[(l-1)n,ln]}),
\end{align}
where $\Lambda(h|_I)$ is the energy gap of the restriction of $h$ to the interval $I$, $\bar h$ is the coarse-graining of $h$ at scale $l$ and $\|D_n(\bar h)\|_{3/2}\le1$ for every fixed $n,\bar h$. 

\medskip

To estimate the second r.~h.~s.~term, we note that the sum of the local energy gaps is bounded by the total energy gap, namely
\begin{align}\label{mb12}
(L/l)^{-1}\sum_{n=1}^{L/l}\Lambda(h|_{[(l-1)n,ln]})\le\Lambda(h),
\end{align}
which is a strengthening of the observation \eqref{m45}. In order to make the first r.~h.~s.~term independent on $\bar h$, we want to apply Lemma~\ref{L:stoch-bis}. As in the argument that led to Remark~\ref{rmk1-bis}, we have that for every configuration $h$
\begin{align*}
\bar h\in\mathcal{N}_\nu\quad\mbox{with}\quad\nu\in\mathbb{N},~\nu\lesssim1+\max_{l\le\rho<L}D(h_\rho)/L.
\end{align*}
Using H\"older's inequality and applying \eqref{mb10} to $A_n(\bar h):=D_n(\bar h)^{3/2}$, we obtain that there exists $\bar D_l$ with $\|\bar D_l\|_{3/2}\le 1$ such that
\begin{align}\label{mb11}
(L/l)^{-1}\sum_{n=1}^N\bar D_n(\bar h)&\le\big((L/l)^{-1}\sum_{n=1}^N\bar D_n(\bar h)^{3/2}\big)^{2/3}\nonumber\\&\lesssim\ln^{2/3}(e+\max_{l\le\rho<L}D(h_\rho)/L)+(L/l)^{-2/3}\bar D_l
\end{align}
Inserting \eqref{mb11} and \eqref{mb12} into \eqref{mb13} at taking the maximum over $l$, we obtain
\begin{align*}
\max_l D(h_l)/L\lesssim\ln^{2/3}(e+\max_{l}D(h_l)/L)+\sum_{l}(L/l)^{-2/3}\bar D_l+\Lambda(h).
\end{align*}
To conclude, we absorb the first r.~h.~s.~term into the l.~h.~s.~and note that
\begin{align*}
\bar D:=\sum_l(L/l)^{-2/3}\bar D_l
\end{align*}
is independent on $h$ and satisfies $\|\bar D\|_{3/2}\lesssim1$.
\qed

\end{appendix}


\begin{thebibliography}{alpha} 

\bibitem{AW}
    \newblock Michael Aizenman and Jan Wehr.
	\newblock Rounding effects of quenched randomness on first-order phase transitions.
	\newblock {\emph{Comm. Math. Phys.}} \textbf{130} (1990), no. 3, 489--528.

\bibitem{ACO}
 	\newblock Giovanni Alberti, Rustum Choksi, Felix Otto.
 	\newblock Uniform energy distribution for an isoperimetric problem with
 	long-range interactions.
 	\newblock \emph{Journal of the American Mathematical Society}, {\bf 22}, 2,
 	569--605, (2009).

\bibitem{AKT}
    \newblock Mikl{\'o}s Ajtai, J{\'a}nos Koml{\'o}s, G{\'a}bor Tusn{\'a}dy.
    \newblock On optimal matchings.
    \newblock \emph{Combinatorica}, {\bf 4}, 4, 259--264 (1984), Springer-Verlag Berlin/Heidelberg.

\bibitem{AS}
 	\newblock Scott Armstrong, Charles Smart.
 	\newblock Quantitative stochastic homogenization of convex integral functionals.
 	\newblock \emph{Annales scientifiques de l'Ecole normale sup{\'e}rieure},
 	{\bf 49}, 2, 423--481, (2016).

\bibitem{AST}
    \newblock Luigi Ambrosio, Federico Stra, Dario Trevisan.
    \newblock A PDE approach to a 2-dimensional matching problem.
    \newblock \emph{Probability Theory and Related Fields}, {\bf 173}, 1, 433--477 (2019), Springer.

\bibitem{B}
	\newblock Vladimir I.\ Bogachev.
	\newblock {\emph{Gaussian measures}}.
	\newblock Mathematical Surveys and Monographs \textbf{62}.
	\newblock American Mathematical Society, Providence, RI, 1998.

\bibitem{BSS}
	\newblock Riddhipratim Basu, Vladas Sidoravicius, and Allan Sly.
	\newblock Rotationally invariant first passage percolation: Breaking the $
	n/\log n $ variance barrier.
	\newblock \emph{arXiv preprint} {2604.01214}, 2026.

\bibitem{CLPS}
    \newblock Sergio Caracciolo, Carlo Lucibello, Giorgio Parisi, Gabriele Sicuro.
    \newblock Scaling hypothesis for the Euclidean bipartite matching problem.
    \newblock \emph{Physical Review E}, {\bf 90}, 1, 012118 (2014), APS.

\bibitem{COP}
    \newblock Xiaopeng Cheng, Felix Otto and Matteo Palmieri.
    \newblock Asymptotics of a planar isoperimetric problem with a white-noise volume term.
    \newblock \emph{arXiv preprint}, 2026.

\bibitem{COPW}
	\newblock Xiaopeng Cheng, Felix Otto, Matteo Palmieri and Arthur Wachtel.
    \newblock A geometrically linear approximation of min-max optimal matching.
	\newblock \emph{arXiv preprint}, 2026.

\bibitem{DOV}
	\newblock Duncan Dauvergne, Janosch Ortmann, and B{\'a}lint Vir{\'a}g.
	\newblock The directed landscape.
	\newblock {\emph{Acta Mathematica}} \textbf{229} (2022), no. 2, 201--285.

\bibitem{DEHP} 
	\newblock Barbara Dembin, Dor Elboim, Daniel Hadas, and Ron Peled.
	\newblock Minimal surfaces in random environment.
	\newblock \emph{arXiv preprint} {2401.06768}, 2024.
	
\bibitem{DEP} 
	\newblock Barbara Dembin, Dor Elboim, and Ron Peled.
	\newblock Minimal surfaces in strongly correlated random environments.
	\newblock \emph{arXiv preprint} {2504.10379}, 2025.
    
\bibitem{DW}
	\newblock Jian Ding and Mateo Wirth.
	\newblock Correlation length of the two-dimensional random field Ising model via greedy lattice animal.
	\newblock {\emph{Duke Math. J.}} \textbf{172} (2023), no. 9, 1781--1811.

\bibitem{DWARXIV}
	\newblock Jian Ding and Mateo Wirth.
	\newblock Correlation length of the two-dimensional random field Ising model via greedy lattice animal.
	\newblock \emph{arXiv preprint} {2011.08768v3}, 2022.

\bibitem{G}
	\newblock Shirshendu Ganguly.
	\newblock \emph{private communication}.

\bibitem{GGN}
	\newblock Shirshendu Ganguly, Victor Ginsburg, and Kyeongsik Nam.
	\newblock Last passage percolation in hierarchical environments.
	\newblock \emph{arXiv preprint} {2411.08018}, 2024.

\bibitem{GHO}
    \newblock Michael Goldman, Martin Huesmann, Felix Otto.
    \newblock Almost sharp rates of convergence for the average cost and displacement in the optimal matching problem.
    \newblock \emph{The Abel Symposium}, 93--103 (2023), Springer.

\bibitem{LS}
	\newblock Thomas Leighton and Peter Shor.
	\newblock Tight bounds for minimax grid matching with applications to the average case analysis of algorithms.
	\newblock \emph{Combinatorica} \textbf{9}  (1989), 161–187.

\bibitem{OPW}
	\newblock Felix Otto, Matteo Palmieri, and Christian Wagner.
	\newblock On minimizing curves in a Brownian potential.
	\newblock \emph{Probability Theory and Related Fields}, 1--62 (2026), Springer.

\bibitem{OV}
 	\newblock Felix Otto, Thomas Viehmann.
 	\newblock Domain branching in uniaxial ferromagnets: asymptotic behavior of the energy.
 	\newblock \emph{Calculus of variations and partial differential equations},
 	{\bf 38}, 1, 135--181 (2010), Springer.

\bibitem{RW}
	\newblock Tobias Ried, and Christian Wagner.
	\newblock Large scale regularity and correlation length for almost length-minimizing random curves in the plane.
	\newblock \emph{arXiv preprint} {2412.17625}, 2024.

\bibitem{T22}
	\newblock Michel Talagrand.
	\newblock {\emph{Upper and lower bounds for stochastic processes:
	decomposition theorems}}.
	\newblock Vol. \textbf{60}.
	\newblock Springer Nature, 2022.

\end{thebibliography}
\end{document}